\RequirePackage{amsmath}
\documentclass[smallextended]{svjour3}
\usepackage{lmodern}

\usepackage{amsthm}

\PassOptionsToPackage{numbers, compress}{natbib}

\usepackage[utf8]{inputenc} 
\usepackage[T1]{fontenc}    
\usepackage{url}            
\usepackage{booktabs}       
\usepackage{amsfonts}       
\usepackage{nicefrac}       
\usepackage{color}

\usepackage{amssymb,amsfonts,bm,mathtools}

\newcommand{\R}{\mathbb{R}}
\newcommand{\E}{\mathbb{E}}
\newcommand{\PP}{\mathbb{P}}

\renewcommand{\P}{\mathbb{P}}
\newcommand{\cP}{\mathcal{P}}
\newcommand{\cF}{\mathcal{F}}
\newcommand{\cH}{\mathcal{H}}
\newcommand{\veps}{\varepsilon}

\newcommand{\bfx}{{\boldsymbol{x}}}
\newcommand{\bfp}{{\boldsymbol{p}}}
\newcommand{\bftheta}{{\boldsymbol{\theta}}}

\newcommand{\bfphi}{{\boldsymbol{\phi}}}
\newcommand{\bfeta}{{\boldsymbol{\eta}}}
\newcommand{\bfF}{{\boldsymbol{F}}}
\newcommand{\bfG}{{\boldsymbol{G}}}
\newcommand{\bfdelta}{{\boldsymbol{\delta}}}
\newcommand{\bfI}{{\boldsymbol{I}}}

\DeclareMathOperator*{\argmax}{arg\,max}

\DeclareMathOperator*{\esssup}{ess\,sup}

\title{A Mean-Field Optimal Control Formulation of Deep Learning}

\author{Weinan E, Jiequn Han, Qianxiao Li}

\institute{
  Weinan E \at
  Princeton University,
  Princeton, NJ 08544, USA,\\
  Beijing Institute of Big Data Research and Peking University,
  Beijing, China 100871
  \and
  Jiequn Han \at
  Princeton University, Princeton, NJ 08544, USA
  \and
  Qianxiao Li \at
  Institute of High Performance Computing,
  Agency for Science, Technology and Research.\\
  1 Fusionopolis Way, Connexis North, Singapore 138632
}

\date{}

\begin{document}

\maketitle

\begin{abstract}
Recent work linking deep neural networks and dynamical systems opened up new avenues to analyze deep learning. In particular, it is observed that new insights can be obtained by recasting deep learning as an optimal control problem on difference or differential equations. However, the mathematical aspects of such a formulation have not been systematically explored. This paper introduces the mathematical formulation of the population risk minimization problem in deep learning as a mean-field optimal control problem. Mirroring the development of classical optimal control, we state and prove optimality conditions of both the Hamilton-Jacobi-Bellman type and the Pontryagin type. These mean-field results reflect the probabilistic nature of the learning problem. In addition, by appealing to the mean-field Pontryagin's maximum principle, we establish some quantitative relationships between population and empirical learning problems. This serves to establish a mathematical foundation for investigating the algorithmic and theoretical connections between optimal control and deep learning.
\end{abstract}

\section{Introduction}
\label{sec:intro}

Deep learning~\cite{bengio2009learning,lecun2015deep,goodfellow2016deep} has become a primary tool in many modern machine learning tasks, such as image classification and segmentation. Consequently, there is a pressing need to provide a solid mathematical framework to analyze various aspects of deep neural networks. The recent line of work on linking dynamical systems, optimal control and deep learning has suggested such a candidate~\cite{e2017proposal,li2017maximum,li2018optimal,haber2017stable,chang2017multi,chang2017reversible,lu2017beyond,sonoda2017double,li2017deep,chen2018neural}. In this view, ResNet~\cite{he2016deep} can be regarded as a time-discretization of a continuous-time dynamical system. Learning (usually in the empirical risk minimization form) is then recast as an optimal control problem, from which novel algorithms~\cite{li2017maximum,li2018optimal} and network structures~\cite{haber2017stable,chang2017multi,chang2017reversible,lu2017beyond} can be designed. An attractive feature of this approach is that, the compositional structure, which is widely considered the essence of deep neural networks is explicitly taken into account in the time-evolution of the dynamical systems. 

While most prior work on the dynamical systems viewpoint of deep learning have focused on algorithms and network structures, this paper aims to study the fundamental mathematical aspects of the formulation. Indeed, we show that  the most general formulation of the population risk minimization problem can be regarded as a {\it mean-field optimal control problem}, in the sense that the optimal control parameters (or equivalently, the trainable weights) depend on the population distribution of input-target pairs. 
Our task is then to analyze the mathematical properties of this mean-field control problem. Mirroring the development of classical optimal control, we will proceed in two parallel, but inter-connected ways, namely the dynamic programming formalism and the maximum principle formalism. 

The paper is organized as follows. We discuss related work in Sec.~\ref{sec:related_work} and introduce the basic mean-field optimal control formulation of deep learning in Sec.~\ref{sec:formulation}. In Sec.~\ref{sec:HJB_derivation}, following the classical dynamic programming approach~\cite{bellman2013dynamic}, we introduce and study the properties of a value function for the mean-field control problem whose state space is an appropriate Wasserstein space of probability measures. By defining an appropriate notion of derivative with respect to probability measures, we show that the value function is related to solutions of an infinite dimensional Hamilton-Jacobi-Bellman (HJB) partial differential equation. With the concept of viscosity solutions~\cite{crandall1983viscosity}, we show in Sec.~\ref{sec:HJB_vicosity} that the HJB equation admits a unique viscosity solution and completely characterize the optimal loss function and the optimal control policy of the mean-field control problem. This establishes a concrete link between the learning problem viewed as a variational problem and the Hamilton-Jacobi-Bellman equation that is associated with the variational problem. It should be noted the essential ideas in the proof of Sec.~\ref{sec:HJB_derivation} and~\ref{sec:HJB_vicosity} are not new, but we present our simplified treatment for this particular setting.

Next, in Sec.~\ref{sec:mean_field_pmp}, we develop the more local theory based on the Pontryagin's maximum principle (PMP)~\cite{pontryagin1987mathematical}. We state and prove a mean-field version of the classical PMP that provides necessary conditions for optimal controls. Further, we study situations when the mean-field PMP admits a unique solution, which then imply that it is also sufficient for optimality, provided an optimal solution exists. We will see in Sec.~\ref{sec:small_T_uniqueness} that compared with the HJB approach, this further requires the fact that the time horizon of the learning problem is small enough. Finally, in Sec.~\ref{sec:error_analysis} we study the relationship between the population risk minimization problem (cast as a mean-field control problem and characterized by a mean-field PMP) and its empirical risk minimization counter-part (cast as a classical control problem and characterized by a classical, sampled PMP). We prove that under appropriate conditions for every stable solution of the mean-field PMP, with high probability there exist close-by solutions of the sampled PMP, and the latter converge in probability to the former, with explicit error estimates on both the distance between the solutions and the distance between their loss function values. This provides a type of \emph{a priori} error estimate that has implications on the generalization ability of neural networks, which is an important and active area of machine learning research. 

Note that it is not the purpose of this paper to prove the sharpest estimates under the most general conditions, thus we have taken the most convenient but reasonable assumptions and the results presented could be sharpened with more technical details. In each section from Sec.~\ref{sec:HJB_derivation} to Sec.~\ref{sec:error_analysis}, we first present the mathematical results, and then discuss the related implications in deep learning. Furthermore, in this work we shall focus our analysis on the continuous idealization of deep residual networks, but we believe that much of the analysis presented also carry over to the discrete domain (i.e.\,discrete layers).

\section{Related work}
\label{sec:related_work}
The connection between back-propagation and optimal control of dynamical systems is known since the earlier works on control and deep learning~\cite{bryson1975applied,athans2013optimal,le1988theoretical}. 
Recently, the dynamical systems approach to deep learning was proposed in~\cite{e2017proposal} and explored in the direction of training algorithms based on the PMP and the method of successive approximations~\cite{li2017maximum,li2018optimal}. In another vein, there are also studies on the continuum limit of neural networks~\cite{sonoda2017double,li2017deep} and on designing network architectures for deep learning~\cite{haber2017stable,chang2017multi,chang2017reversible,lu2017beyond} based on dynamical systems and differential equations. 
Instead of analysis of algorithms or architectures, the present paper focuses on the mathematical aspects of the control formulation itself, and develops a mean-field theory that characterize the optimality conditions and value functions using both PDE (HJB) and ODE (PMP) approaches. The over-arching goal is to develop the mathematical foundations of the optimal control formulation of deep learning. 

In the control theory literature, mean-field optimal control is an active area of research. Many works on mean-field games~\cite{lasry2007mean,huang2006large,gueant2011mean,bensoussan2013mean}, the control of McKean-Vlasov systems~\cite{lauriere2014dynamic,pham2017dynamic,pham2018bellman}, and the control of Cucker-Smale systems~\cite{caponigro2015sparse,fornasier2014mean,bongini2017mean} focus on deriving the limiting partial differential equations that characterize the optimal control as the number of agents goes to infinity. This is akin to the theory of the propagation of chaos~\cite{sznitman1991topics}. Meanwhile there are also works discussing the stochastic maximum principle for stochastic differential equations of mean-field type~\cite{andersson2011maximum,buckdahn2011general,carmona2015forward}.
The present paper differs from all previous works in two aspects. First, in the context of continuous-time deep learning, the problem differs from these previous control formulations as the source of randomness are coupled input-target  pairs (the latter determines the terminal loss function, which can now be regarded as a random function). On the other hand, a simplifying feature in our case is that the dynamics, given the input-target pair, are otherwise deterministic. Second, the dynamics of each random realization are independent of the distribution law of the population, and are coupled only through the shared control parameters. This is to be contrasted with optimal control of McKean-Vlasov dynamics~\cite{carmona2015forward,pham2017dynamic,pham2018bellman} or mean-field games~\cite{lasry2007mean,huang2006large,gueant2011mean,bensoussan2013mean}, where the population law directly enters the dynamical equations (and not just through the shared control). Thus, in this sense our dynamical equations are much simpler to analyze. Consequently, although some of our results can be deduced from more general mean-field analysis in the control literature, here we will present simplified derivations tailored to our setting, Note also that there are neural network structures (e.g.\,batch-normalization) that can be considered to have explicit mean-field dynamics, and we defer this discussion to Sec.~\ref{sec:conclusion}. 

\section{From ResNets to mean-field optimal control}
\label{sec:formulation}

Let us now present the optimal control formulation of deep learning as introduced in~\cite{e2017proposal,li2017maximum,li2018optimal}. In the simplest form, the feed-forward propagation in a $T$-layer residual network can be represented by the difference equations
\begin{align}
    x_{t+1} = x_{t} + f(x_{t}, \theta_{t}), \qquad t=0,\dots,T-1.
    \label{eq:discrete_forward}
\end{align}
where $x_0$ is the input (image, time-series, etc.) and $x_T$ is the final output. The final output is then compared with some target  $y_0$ corresponding to $x_0$ via some loss function. The goal of learning is to tune the trainable parameters $\theta_0,\dots,\theta_{T-1}$ such that $x_T$ is close to $y_0$. The only change in the continuous-time idealization of deep residual learning, which we will subsequently focus on, is that instead of the difference equation~\eqref{eq:discrete_forward}, the forward dynamics is now a differential equation. Let us now introduce this formulation more precisely. 

Let $(\Omega,\cF,\P)$ be a fixed and sufficiently rich probability space so that all subsequently required random variables can be constructed. Suppose $x_0\in\R^d$ and $y_0\in\R^l$ are random variables jointly distributed according to $\mu_0 \coloneqq \P_{(x_0,y_0)}$ (hereafter, for each random variable $X$ we denote its distribution or law by $\P_X$). This represents the distribution of the input-target pairs, which we assume can be embedded in Euclidean spaces. 
Consider a set of admissible controls or training weights $\Theta\subseteq \R^m$. In typical deep learning, $\Theta$ is taken as the whole space $\R^m$, but here we consider the more general case where $\Theta$ can be constrained. 
Fix $T>0$ (network ``depth'') and let $f$ (feed-forward dynamics), $\Phi$ (terminal loss function) and $L$ (regularizer) be functions
\begin{align*}
	f: \R^d \times \Theta \rightarrow \R^d, \quad
    \Phi: \R^d \times \R^l \rightarrow \R, \quad
    L: \R^d \times \Theta \rightarrow \R.
\end{align*}
We define the state dynamics as the ordinary differential equation (ODE)
\begin{align}
	\dot{x}_t & = f(x_t, \theta_t)
    \label{eq:cts_forward}
\end{align}
with initial condition equals to the random variable $x_0$. Thus, this is a stochastic ODE, whose only source of randomness is on the initial condition. Consider the set of essentially bounded measurable controls $L^{\infty}([0,T],\Theta)$. To improve clarity, we will reserve bold-faced letters for path-space quantities. For example, $\bftheta \equiv \{ \theta_t: 0\leq t\leq T \}$. In contrast, variables/functions taking values in finite-dimensional Euclidean spaces are not bold-faced. 

The population risk minimization problem in deep learning can hence be posed as the following \emph{mean-field optimal control problem}
\begin{align}
    \begin{split}
        \inf_{\bftheta \in L^{\infty}([0,T],\Theta)}              
        J(\bftheta) &\coloneqq \E_{\mu_0} 
        \left[    
        \Phi(x_T, y_0)                        
        + \int_{0}^{T} L(x_t, \theta_t) dt 
        \right],\\ 
        & \text{Subject to~\eqref{eq:cts_forward}}. 
    \end{split}
    \label{eq:mean_field_ctrl_prob} 
\end{align}
The term ``mean-field'' highlights the fact that $\bftheta$ is shared by a whole population of input-target  pairs, and the optimal control must depend on the law of the input-target random variables. Strictly speaking, the law of $\bfx$ does not enter the forward equations explicitly (unlike e.g.,\,McKean-Vlasov control~\cite{carmona2015forward}), and hence our forward dynamics are not explicitly in mean-field form. Nevertheless, we will use the term ``mean-field'' to emphasize the dependence of the control on the population distribution.  

In contrast, if we were to perform empirical risk minimization, as is often the case in practice (and is the case analyzed by previous work on algorithms~\cite{li2017maximum,li2018optimal}), we would first draw i.i.d.\,samples $\{x_0^i,y_0^i\}_{i=1}^{N}\sim \mu_0$ and pose the \emph{sampled optimal control problem}
\begin{align}
    \begin{split}
        \inf_{\bftheta \in L^{\infty}([0,T],\Theta)}              
        J_N(\bftheta) &\coloneqq \frac{1}{N} \sum_{i=1}^N 
        \left[    
        \Phi(x^{i}_T, y^i_0)                        
        + \int_{0}^{T} L(x^i_t, \theta_t) dt 
        \right],\\ 
        & \text{Subject to} \qquad \dot{x}^{i}_t = f(x^{i}_t, \theta_t), \qquad i=1,\dots,N.
    \end{split}
    \label{eq:sampled_ctrl_prob} 
\end{align}
Thus, the solutions of sampled optimal control problems are typically random variables. We now focus our analysis on the mean-field problem~\eqref{eq:mean_field_ctrl_prob} and only later in Sec.~\ref{sec:error_analysis} relate it with the sampled problem~\eqref{eq:sampled_ctrl_prob}. 

\subsection*{Additional Notation}
Throughout this paper, we always use $w$ to denote the concatenated $(d+l)$-dimensional variable $(x,y)$ where $x\in\R^d$ and $y\in\R^l$. Correspondingly $\bar{f}(w,\theta)\coloneqq(f(x,\theta),0)$ is the extended $(d+l)$-dimensional feed-forward function, $\bar{L}(w,\theta)\coloneqq L(x,\theta)$ is the extended $(d+l)$-dimensional regularization loss, and $\bar{\Phi}(w)\coloneqq\Phi(x,y)$ still denotes the terminal loss function. We denote by 
$x\cdot y$ the inner product of two Euclidean vectors $x$ and $y$ with the same dimension. The Euclidean norm is denoted by $\Vert\cdot\Vert$ and the absolute value is denoted by $\vert \cdot \vert$. 
Gradient operators on Euclidean spaces are denoted by $\nabla$ with subscripts indicating the variable with which the derivative is taken with. In contrast, we use $D$ to represent the Fr\'echet derivative on Banach spaces. Namely, if $x\in U$ and $F:U\rightarrow V$ is a mapping between two Banach spaces $(U,\Vert\cdot \Vert_U)$ and $(V,\Vert \cdot \Vert_V)$, then $DF(x)$ is defined by the linear operator $DF(x):U \rightarrow V$ s.t. 
\begin{align}
  \label{eq:frechet_deriv_defn}
    r(x,y) \coloneqq \frac{\Vert F(x + y) - F(x) - DF(x)y \Vert_V}{\Vert y \Vert_U} 
    \rightarrow 0,                                                                
    \quad \text{as }                                                    
    \Vert y \Vert_U \rightarrow 0.                                        
\end{align}
For a matrix $A$, we use the symbol $A \preceq 0$ to mean that $A$ is negative semi-definite. 

Let the Banach space $L^\infty([0,T],E)$ be the set of essentially bounded measurable functions from $[0,T]$ to $E$, where $E$ is a subset of a Euclidean space with the usual Lebesgue measure. The norm is $\Vert \bfx \Vert_{L^\infty([0,T],E)} = \esssup_{t\in[0,T]} \Vert x(t) \Vert$, and we shall write for brevity $\Vert\cdot\Vert_{L^\infty}$ in place of $\Vert\cdot\Vert_{L^\infty([0,T],E)}$. In this paper, $E$ is often either $\Theta$ or $\R^d$, and the path-space variables we consider in this paper, such as the controls $\bftheta$, will mostly be defined in this space. 

As this paper introduces a mean-field optimal control approach, we also need some notation for the random variables and their distributions. We use the shorthand $L^2(\Omega, \R^{d+l})$ for $L^2((\Omega,\cF,\P), \R^{d+l})$, the set of $\R^{d+l}$-valued square integrable random variables. We equip this Hilbert space with the norm $\|X\|_{L^2}\coloneqq (\E \| X \| ^2)^{1/2}$ for $X\in L^2(\Omega,\R^{d+l})$. We denote by $\cP_2(\R^{d+l})$ the set of square integrable probability measures on the Euclidean space $\R^{d+l}$. Note that $X\in L^2(\Omega,\R^{d+l})$ if and only if $\P_X \in \cP_2(\R^{d+l})$.
The space $\cP_2(\R^{d+l})$ is regarded as a metric space equipped with the 2-Wasserstein distance
\begin{alignat*}{2}
  W_2(\mu,\nu)\coloneqq\inf\Big\{&   \Big(\int_{\R^{d+l}\times \R^{d+l}}\|w-z\|^2\pi(dw,dz)\Big)^{1/2}\,\Big|\, \\
  &  \pi\in\cP_2(\R^{d+l}\times \R^{d+l})\text{ with marginals } \mu \text{ and } \nu \Big\}\\
  \coloneqq\inf\Big\{&   \|X-Y\|_{L^2}\,\Big|\,X,Y\in L^2(\Omega,\R^{d+l}) \text{ with } \P_X=\mu,\,\P_Y=\nu \Big\}.
\end{alignat*}
For $\mu\in\cP_2(\R^{d+l})$, we also define $\|\mu\|_{L^2}\coloneqq (\int_{\R^{d+l}}\|w\|^2\mu(dw))^{1/2}$. 

Given a measurable function $\psi: \R^{d+l}\rightarrow \R^q$ that is square integrable with respect to $\mu$, we use the notation
\begin{equation*}
\langle\psi(.),\,\mu\rangle \coloneqq \int_{\R^{d+l}}\psi(w)\mu(dw).
\end{equation*}

Now, we introduce some notation for the dynamical evolution of probabilities. Given $\xi \in L^2(\Omega, \R^{d+l})$ and a control process $\bftheta \in L^\infty([0,T],\Theta)$, we consider the following dynamical system for $t\leq s\leq T$:
\begin{equation*}
W_s^{t,\xi,\bftheta}=\xi + \int_{t}^s \bar{f}(W_s^{t,\xi,\bftheta},\theta_t)\,ds.
\end{equation*}
Note that $W_s^{t,\xi,\bftheta}$ is always square integrable given $\bar{f}(w,\theta)$ is Lipschitz continuous with respect to $w$. Let $\mu = \P_{\xi}\in \cP_2(\R^{d+l})$, 
we denote the law of $W_s^{t,\xi,\bftheta}$ for simplicity by
\begin{equation*}
\P_s^{t,\mu,\bftheta}\coloneqq \P_{W_s^{t,\xi,\bftheta}}.
\end{equation*}
This is valid since the law of $W_s^{t,\xi,\bftheta}$ should only depend on the law of $\xi$ and not on the random variable itself. This notation also allow as to write down the flow or semi-group property of the dynamical system as
\begin{equation}
  \P_{s}^{t,\mu,\bftheta} = \P_s^{\hat{t},\P_{\hat{t}}^{t,\mu,\bftheta},\bftheta}, 
  \label{eq:flow_property}
\end{equation}
for all $0\leq t\leq \hat{t} \leq s \leq T,\,\mu\in \cP_2(\R^{d+l}),\,\bftheta \in L^\infty([0,T],\Theta)$.

Finally, throughout the results and proofs, we will use $K$ or $C$ with subscripts as names for generic constants, whose values may change from line to line when there is no need for them to be distinct. In general, these constants may implicitly depend on $T$ and the ambient dimensions $d,m$, but for brevity we omit them in the rest of the paper. 

\section{Mean-field dynamic programming principle and HJB equation}
\label{sec:HJB_derivation} 

We begin our analysis of~\eqref{eq:mean_field_ctrl_prob} by employing the dynamic programming principle and the Hamilton-Jacobi-Bellman formalism. In this approach, the key idea is to define a value function that corresponds to the optimal loss of the control problem~\eqref{eq:mean_field_ctrl_prob}, but under a general starting time and starting state. One can then derive a partial differential equation (Hamilton-Jacobi-Bellman equation, or HJB equation) to be satisfied by such a value function, which characterizes both the optimal loss function value and the optimal control policy of the original control problem. Compared to the classical optimal control case corresponding to empirical risk minimization in learning, here the value function's state argument is no longer a finite-dimensional vector, but an infinite-dimensional object corresponding to the joint distribution of the input-target pair. We shall interpret it as an element of a suitable Wasserstein space. The detailed mathematical definition of this value function and its basic properties are discussed in Subsec.~\ref{sec:value_continuity}.

In the finite-dimensional case, the HJB equation is a classical partial differential equation. In contrast, since the state variables we are dealing with are probability measures rather than Euclidean vectors, we need a concept of derivative with respect to a probability measure, as introduced by Lions in his course at Coll\`ege de France~\cite{lions2012cours}. We give a brief introduction of this concept in Subsec.~\ref{sec:derivative_in_wasserstein} and refer readers to the lecture notes~\cite{cardaliaguet2010notes} for more details. We then present the resulting infinite-dimensional HJB equation in Subsec.~\ref{sec:HJB_wasserstein}. 

Throughout this section and next section (Sec.~\ref{sec:HJB_vicosity}), we assume
\begin{itemize}
\item[(A1)] $f,L,\Phi$ are bounded; $f,L,\Phi$ are Lipschitz continuous with respect to $x$, and the Lipschitz constants of $f$ and $L$ are independent of $\theta$.
\item[(A2)] $\mu_0\in\cP_2(\R^{d+l})$.
\end{itemize}

\subsection{Value function and its properties}
\label{sec:value_continuity}

Adopting the viewpoint of taking probability measures $\mu\in \cP_2(\R^{d+l})$ as state variables, we can define a time-dependent objective functional
\begin{align}
  J(t,\mu,\bftheta) &\coloneqq~
  \E_{(x_t,y_0)\sim \mu} \left[    
    \Phi(x_T, y_0)+ \int_{t}^{T} L(x_t, \theta_t) dt \right]
    \text{ (subject to~\eqref{eq:cts_forward})} \nonumber \\
  &=~\langle \bar{\Phi}(.), \,\P_{T}^{t,\mu,\bftheta} \rangle + \int_{t}^T \langle \bar{L}(.,\theta_s), \,\P_{s}^{t,\mu,\bftheta} \rangle\,ds.
\label{eq:time-mu objective}
\end{align}
The second line in the above is just a rewriting of the first line based on the notation introduced earlier. Here, we abuse the notation $J$ in~\eqref{eq:mean_field_ctrl_prob} for the new objective functional, which now has additional arguments $t,\mu$. Of course, $J(\bftheta)$ in~\eqref{eq:mean_field_ctrl_prob} corresponds to $J(0,\mu_0,\bftheta)$ in~\eqref{eq:time-mu objective}.

The \emph{value function} $v^*(t,\mu)$ is defined as a real-valued function on $[0,T]\times\cP_2(\R^{d+l})$ through
\begin{equation}
\label{eq:value_def}
v^*(t,\mu) = \inf_{\bftheta \in L^\infty([0,T],\Theta)} J(t,\mu,\bftheta).
\end{equation}
If we assume $\bftheta^*$ attains the infimum in \eqref{eq:mean_field_ctrl_prob}, then by definition
\begin{equation*}
  J(\bftheta^*)=v^*(0,\mu_0).
\end{equation*}
The following proposition shows the continuity of the value function.

\begin{proposition}
\label{prop:value_continuity}
The function $(t,\mu)\mapsto J(t,\mu,\bftheta)$ is Lipschitz continuous on $[0,T]\times \cP_2(\R^{d+l})$, uniformly with respect to $\bftheta\in L^\infty([0,T],\Theta)$, and the value function $v^*(t,\mu)$ is Lipschitz continuous on $[0,T]\times \cP_2(\R^{d+l})$.
\end{proposition}

\begin{proof}
  We first establish some elementary estimates based on the assumptions. We {}suppose
  \begin{equation}
    \langle \bar{L}(.,\theta),\,\mu\rangle \leq C.
    \label{eq:L_boundedness}
  \end{equation}
  Let $X,Y\in L_2(\Omega,\R^{d+l})$ such that $\P_X=\mu, \P_Y=\hat{\mu}$, the Lipschitz continuity of $\bar{L}$ gives us
  \begin{align*}
    |\langle \bar{L}(.,\theta),\,\mu\rangle - \langle \bar{L}(.,\theta),\,\hat{\mu}\rangle| 
    = |\E[\bar{L}(X,\theta)-\bar{L}(Y,\theta)]|
    \leq K_L \|X-Y\|_{L^2}.
  \end{align*}
  Note that in the proceeding inequality the left side does not depend on the choice of $X,Y$ while the right side does. Hence we can take the infimum over all the joint choices of $X,Y$ to get
  \begin{align}
    &|\langle \bar{L}(.,\theta),\,\mu\rangle - \langle \bar{L}(.,\theta),\,\hat{\mu}\rangle| \nonumber \\
    \leq &K_L\times \inf\Big\{\|X-Y\|_{L^2}\,\Big|\,X,Y\in L^2(\Omega,\R^{d+l}) \text{ with } \P_X=\mu,\,\P_Y=\nu \Big\} \nonumber \\
    \leq &K_L W_2(\mu,\hat{\mu}).
    \label{eq:L_continuity}
  \end{align}
  The same argument applied to $\bar{\Phi}$ gives us
  \begin{equation}
    |\langle \bar{\Phi}(.),\,\mu\rangle - \langle \bar{\Phi}(.),\,\hat{\mu}\rangle| \leq K_LW_2(\mu, \hat{\mu}).
    \label{eq:Phi_continuity}
  \end{equation}
  For the deterministic ODE
  \begin{equation*}
    \frac{dw_t^{\bftheta}}{dt}=\bar{f}(w_t^{\bftheta},\theta_t), 
    \quad w_0^{\bftheta} = w_0,
  \end{equation*}
  define the induced flow map as
  \begin{equation*}
    h(t,w_0,\bftheta) \coloneqq w_t^{\bftheta}.
  \end{equation*}
  Using Gronwall's inequality with the boundedness and Lipschitz continuity of $\bar{f}$, we know
  \begin{align*}
    &|h(t,w,\bftheta) - h(t,\hat{w},\bftheta)| \leq K_L\|w-\hat{w}\|, \\
    &|h(t,w,\bftheta) - h(\hat{t},w,\bftheta)| \leq K_L|t-\hat{t}|.
  \end{align*}
  Therefore we use the definition of Wasserstein distance to obtain
  \begin{align}
    &W_2(\P_s^{t,\mu, \bftheta}, \P_s^{t,\hat{\mu}, \bftheta}) \nonumber \\
    =\,& \inf\Big\{\|X-Y\|_{L^2}\,\Big|\,X,Y\in L^2(\Omega,\R^{d+l}) \text{ with } \P_X=\P_s^{t,\mu, \bftheta},\,\P_Y=\P_s^{t,\hat{\mu}, \bftheta} \Big\} \nonumber \\
    =\,& \inf\Big\{\|h(s-t,X,\bftheta)-h(s-t,Y,\bftheta)\|_{L^2}\,\Big|\, \nonumber \\
    &\hphantom{\inf\Big\{} X,Y\in L^2(\Omega,\R^{d+l}) \text{ with } \P_X=\mu,\,\P_Y=\hat{\mu} \Big\} \nonumber \\
    \leq \,& \inf\Big\{K_L\|X-Y\|_{L^2}\,\Big|\,X,Y\in L^2(\Omega,\R^{d+l}) \text{ with } \P_X=\mu,\,\P_Y=\hat{\mu} \Big\} \nonumber \\
    =\,&  K_LW_2(\mu,\hat{\mu})   \label{eq:W2_continuity1}
  \end{align}
  and similarly
  \begin{align}
    \label{eq:W2_continuity2}
    W_2(\P_s^{t,\mu, \bftheta}, \mu) &\leq K_L |s-t|.
  \end{align}
  The flow property~\eqref{eq:flow_property} and estimates~\eqref{eq:W2_continuity1}, \eqref{eq:W2_continuity2} together give us
  \begin{align}
    W_2(\P_s^{t,\mu,\bftheta},\P_s^{\hat{t},\hat{\mu},\bftheta}) 
    &= W_2(\P_s^{\hat{t},\P_{\hat{t}}^{t,\mu,\bftheta},\bftheta},\P_s^{\hat{t},\hat{\mu},\bftheta})  \nonumber \\
    & \leq K_LW_2(\P_{\hat{t}}^{t,\mu,\bftheta},\hat{\mu}) \nonumber \\
    & \leq K_L(|t-\hat{t}| + W_2(\mu,\hat{\mu})).
    \label{eq:W2_continuity3}
  \end{align}
  Now for all $0\leq t\leq \hat{t}\leq T$, $\mu,\hat{\mu}\in\cP_2(\R^{d+l})$, $\bftheta\in L^\infty([0,T],\Theta)$, we employ~\eqref{eq:L_boundedness},~\eqref{eq:L_continuity},~\eqref{eq:Phi_continuity}, and~\eqref{eq:W2_continuity3} to obtain
  \begin{align*}
    \,&|J(t,\mu,\bftheta)-J(\hat{t},\hat{\mu},\bftheta)| \\
    \leq\, & \int_{t}^{\hat{t}} |\langle \bar{L}(.,\theta_s), \,\P_{s}^{t,\mu,\bftheta} \rangle|\,ds + 
    \int_{\hat{t}}^{T} |\langle \bar{L}(.,\theta_s), \,\P_{s}^{t,\mu,\bftheta} \rangle - \langle \bar{L}(.,\theta_s), \,\P_{s}^{\hat{t},\hat{\mu},\bftheta} \rangle|\,ds  \\
    \,&+ |\langle \bar{\Phi}(.), \,\P_{T}^{t,\mu,\bftheta} \rangle - \langle \bar{\Phi}(.), \,\P_{T}^{\hat{t},\hat{\mu},\bftheta} \rangle| \\
    \leq\, & C|\hat{t}-t| + K_L\sup_{\hat{t}\leq s\leq T}W_2(\P_s^{t,\mu,\bftheta},\P_s^{\hat{t},\hat{\mu},\bftheta}) \\
    \leq\, & K_L(|t-\hat{t}| + W_2(\mu,\hat{\mu})),
  \end{align*} 
  which gives us the desired Lipschitz continuity property.

  Finally, combining the fact that
  \begin{eqnarray*}
    &|v^*(t,\mu)-v^*(\hat{t},\hat{\mu})|\leq 
    \sup_{\bftheta\in L^\infty([0,T],\Theta)}|J(t,\mu,\bftheta)-J(\hat{t},\hat{\mu},\bftheta)|, \\
    &\forall~t,\hat{t}\in[0,T], \,\mu,\hat{\mu}\in \cP_2(\R^{d+l}),
  \end{eqnarray*}
  and $J(t,\mu,\bftheta)$ is Lipschitz continuous at $(t,\mu)\in[0,T]\times \cP_2(\R^{d+l})$, uniformly with respect to $\bftheta\in L^\infty([0,T],\Theta)$, we deduce that the value function $v^*(t,\mu)$ is Lipschitz continuous on $[0,T]\times \cP_2(\R^{d+l})$.
\end{proof}

The important observation we now make is that the value function satisfies a recursive relation. This is known as the \emph{dynamic programming principle}, which forms the basis of deriving the Hamilton-Jacobi-Bellman equation. 
Intuitively, the dynamic programming principle states that for any optimal trajectory, starting from any intermediate state in the trajectory, the remaining trajectory must again be optimal, starting from that time and state. We now state and prove this intuitive statement precisely. 

\begin{proposition}{(Dynamic programming principle)}
  \label{prop:DPP}
  For all $0\leq t \leq \hat{t} \leq T$, $\mu\in \cP_2(\R^{d+l})$, we have
  \begin{equation}
    \label{eq:DPP_v}
    v^*(t,\mu) = \inf_{\bftheta\in L^\infty([0,T],\Theta)}\Big[ \int_{t}^{\hat{t}}\langle \bar{L}(.,\theta_s), \,\P_{s}^{t,\mu,\bftheta} \rangle\,ds + v^*(\hat{t}, \P_{\hat{t}}^{t,\mu,\bftheta})\Big].
  \end{equation}
\end{proposition}

\begin{proof}
  The proof is elementary as in the context of deterministic control problem. We provide it as follows for completeness.

  1). Given fixed $t,\hat{t},\mu$ and any $\bftheta^1 \in L^\infty([0,T],\Theta)$, we consider the probability measure $\P_{\hat{t}}^{t,\mu,\bftheta^1}$. Fix $\veps>0$ and by definition of value function~\eqref{eq:value_def} we can pick $\bftheta^2 \in L^\infty([0,T],\Theta)$ satisfying
  \begin{equation}
    v^*(\hat{t},\P_{\hat{t}}^{t,\mu,\bftheta^1}) + \veps \geq \langle \bar{\Phi}(.), \,\P_{T}^{\hat{t},\P_{\hat{t}}^{t,\mu,\bftheta^1},\bftheta^2} \rangle + \int_{\hat{t}}^T \langle \bar{L}(.,\theta^2_s), \, \P_{s}^{\hat{t},\P_{\hat{t}}^{t,\mu,\bftheta^1},\bftheta^2} \rangle\,ds.
    \label{eq:DPP_proof_inequal1}
  \end{equation}
  Now consider the control process $\hat{\bftheta}$ defined as
  \begin{equation*}
    \hat{\theta}_s = \mathbf{1}_{\{s < \hat{t}\}}\theta^1_s + 
    \mathbf{1}_{\{s\geq \hat{t}\}}\theta^2_s. 
  \end{equation*}
  Thus we can use~\eqref{eq:DPP_proof_inequal1} and flow property~\eqref{eq:flow_property} to deduce
  \begin{align*}
    &~v^*(t,\mu) \\
    \leq& \int_{t}^{T}\langle \bar{L}(.,\hat{\theta}_s), \,\P_{s}^{t,\mu,\hat{\bftheta}} \rangle\,ds 
    + \langle \bar{\Phi}(.), \,\P_{T}^{t,\mu,\hat{\bftheta}} \rangle \\
    = & \int_{t}^{\hat{t}}\langle \bar{L}(.,\hat{\theta}_s), \,\P_{s}^{t,\mu,\hat{\bftheta}} \rangle\,ds
    + \int_{\hat{t}}^{T}\langle \bar{L}(.,\hat{\theta}_s), \,\P_{s}^{t,\mu,\hat{\bftheta}} \rangle\,ds
    + \langle \bar{\Phi}(.), \,\P_{T}^{t,\mu,\hat{\bftheta}} \rangle \\
    = & \int_{t}^{\hat{t}}\langle \bar{L}(.,\hat{\theta}_s), \,\P_{s}^{t,\mu,\hat{\bftheta}} \rangle\,ds
    + \int_{\hat{t}}^{T}\langle \bar{L}(.,\theta^2_s), \, \P_{s}^{\hat{t},\P_{\hat{t}}^{t,\mu,\bftheta^1},\bftheta^2} \rangle\,ds
    + \langle \bar{\Phi}(.), \,\P_{T}^{\hat{t},\P_{\hat{t}}^{t,\mu,\bftheta^1},\bftheta^2} \rangle \\
    \leq & \int_{t}^{\hat{t}}\langle \bar{L}(.,\hat{\theta}_s), \,\P_{s}^{t,\mu,\hat{\bftheta}} \rangle\,ds 
    + v^*(\hat{t},\P_{\hat{t}}^{t,\mu,\bftheta^1}) + \veps \\
    = & \int_{t}^{\hat{t}}\langle \bar{L}(.,\theta^1_s), \,\P_{s}^{t,\mu,\bftheta^1} \rangle\,ds 
    + v^*(\hat{t},\P_{\hat{t}}^{t,\mu,\bftheta^1}) + \veps.
  \end{align*}
  As $\bftheta^1$ and $\veps$ are both arbitrary, we have
  \begin{equation*}
    v^*(t,\mu) \leq \inf_{\bftheta\in L^\infty([0,T],\Theta)}\Big[ \int_{t}^{\hat{t}}\langle \bar{L}(.,\theta_s), \,\P_{s}^{t,\mu,\bftheta} \rangle\,ds + v^*(\hat{t}, \P_{\hat{t}}^{t,\mu,\bftheta})\Big].
  \end{equation*}

  2). Fix $\veps>0$ again and we choose by definition $\bftheta^3 \in L^\infty([0,T],\Theta)$ such that 
  \begin{equation*}
    v^*(t,\mu) + \veps \geq \int_{t}^{T}\langle \bar{L}(.,\theta_s), \,\P_{s}^{t,\mu,\bftheta^3} \rangle\,ds + \langle \bar{\Phi}(.), \,\P_{T}^{t,\mu, \bftheta^3} \rangle.
  \end{equation*}
  Using the flow property~\eqref{eq:flow_property} and the definition of the value function again gives us the estimate
  \begin{align*}
    &~v^*(t,\mu) + \veps \\
    \geq & \int_{t}^{T}\langle \bar{L}(.,\theta^3_s), \,\P_{s}^{t,\mu,\bftheta^3} \rangle\,ds + \langle \bar{\Phi}(.), \,\P_{T}^{t,\mu, \bftheta^3} \rangle \\
    = & \int_{t}^{\hat{t}}\langle \bar{L}(.,\theta^3_s), \,\P_{s}^{t,\mu,\bftheta^3} \rangle\,ds 
    + \int_{\hat{t}}^{T}\langle \bar{L}(.,\theta^3_s), \,\P_{s}^{\hat{t},\P_{\hat{t}}^{t,\mu,\bftheta^3},\bftheta^3} \rangle\,ds
    + \langle \bar{\Phi}(.), \,\P_{T}^{\hat{t},\P_{\hat{t}}^{t,\mu,\bftheta^3},\bftheta^3} \rangle \\
    \geq & \int_{t}^{\hat{t}}\langle \bar{L}(.,\theta^3_s), \,\P_{s}^{t,\mu,\bftheta^3} \rangle\,ds 
    + v^*(\hat{t}, \P_{\hat{t}}^{t,\mu,\bftheta^3}) \\
    \geq & \inf_{\bftheta\in L^\infty([0,T],\Theta)}\Big[ \int_{t}^{\hat{t}}\langle \bar{L}(.,\theta_s), \,\P_{s}^{t,\mu,\bftheta} \rangle\,ds + v^*(\hat{t}, \P_{\hat{t}}^{t,\mu,\bftheta})\Big].
  \end{align*}
  Hence we deduce
  \begin{equation*}
    v^*(t,\mu) \geq \inf_{\bftheta\in L^\infty([0,T],\Theta)}\Big[ \int_{t}^{\hat{t}}\langle \bar{L}(.,\theta_s), \,\P_{s}^{t,\mu,\bftheta} \rangle\,ds + v^*(\hat{t}, \P_{\hat{t}}^{t,\mu,\bftheta})\Big].
  \end{equation*}
  Combining the inequalities in the two parts completes the proof.
\end{proof}

\subsection{Derivative and Chain Rule in Wasserstein Space}
\label{sec:derivative_in_wasserstein}
In classical finite-dimensional optimal control, the HJB equation can be formally derived from the dynamic programming principle by a Taylor expansion of the value function with respect to the state vector. However, in the current formulation, the state is now a probability measure. To derive the corresponding HJB equation in this setting, it is essential to define a notion of derivative of the value function with respect to a probability measure. The basic idea to achieve this is to take probability measures on $\R^{d+l}$ as laws of $\R^{d+l}$-valued random variables on the probability space $(\Omega,\cF,\P)$ and then use the corresponding Banach space of random variables to define derivatives. This approach is more extensively outlined in~\cite{cardaliaguet2010notes}. 

Concretely, let us take any function $u: \cP_2(\R^{d+l})\rightarrow \R$. We now lift it into its ``extension'' $U$, a function defined on $L^2(\Omega,\R^{d+l})$ by
\begin{equation}
U(X)=u(\P_X),\quad \forall X\in L^2(\Omega,\R^{d+l}).
\end{equation}
We say $u$ is $C^1(\cP_2(\R^{d+l}))$ if the lifted function $U$ is Fr\'echet differentiable with continuous derivatives.
Since we can identify $L^2(\Omega,\R^{d+l})$ with its dual space, if the Fr\'echet derivative $DU(X)$ exists, by Riesz' theorem one can view it as an element of $L^2(\Omega,\R^{d+l})$:
\begin{equation*}
DU(X)(Y)=\E[DU(X)\cdot Y],\quad \forall Y\in L^2(\Omega,\R^{d+l}).
\end{equation*}
The important result one can prove is that the law of $DU(X)$ does not depend on $X$ but only on the law of $X$. Accordingly we have the representation
\begin{equation*}
  DU(X) = \partial_\mu u(\P_X)(X),
\end{equation*}
for some function $\partial_\mu u(\P_X): \R^{d+l}\rightarrow \R^{d+l}$, which is called derivative of $u$ at $\mu = \P_X$. 
Moreover, we know $\partial_\mu u(\mu)$ is square integrable with respect to $\mu$.

We next need a chain rule defined on $\cP_2(\R^{d+l})$. Consider the dynamical system
\begin{equation*}
W_t=\xi + \int_{0}^t \bar{f}(W_s)\,ds,\quad \xi\in L^2(\Omega,\R^{d+l}),
\end{equation*}
and $u\in\mathcal{C}^1(\cP_2(\R^{d+l}))$. Then, for all $t\in[0,T]$, we have
\begin{equation}
\label{eq:funcito}
u(\P_{W_t})=u(\P_{W_0})+\int_0^t\langle \partial_{\mu} u(\P_{W_s})(.)\cdot \bar{f}(.),\,\P_{W_s}\rangle \,ds,
\end{equation}
or equivalently its lifted version
\begin{equation}
\label{eq:funcito_lift}
U({W_t})=U({W_0})+\int_0^t \E[DU(W_s)\cdot \bar{f}(W_s)] \,ds.
\end{equation}


\subsection{HJB equation in Wasserstein Space}
\label{sec:HJB_wasserstein}
Guided by the dynamic programming principle~\eqref{eq:DPP_v} and formula~\eqref{eq:funcito}, we are ready to formally derive the associated HJB equation as follows. Let $\hat{t}=t+\delta t$ with $\delta t$ being small. By performing a formal Taylor series expansion of~\eqref{eq:DPP_v}, we have
\begin{align*}
  0 & = \inf_{\bftheta\in L^\infty([0,T],\Theta)}\Big[ v^*(t+\delta t, \P_{t+\delta t}^{t,\mu,\bftheta}) - v^*(t,\mu) + \int_{t}^{t+\delta t}\langle \bar{L}(.,\theta_s), \,\P_{s}^{t,\mu,\bftheta} \rangle\,ds\Big] \\
  & \approx \inf_{\bftheta\in L^\infty([0,T],\Theta)}\Big[ \partial_t v(t,\mu)\delta t + \int_{t}^{t+\delta t}\langle \partial_\mu v(t,\mu)(.)\cdot \bar{f}(.,\theta) + \bar{L}(.,\theta_s), \,\mu \rangle\,ds\Big] \\
  & \approx \delta t \inf_{\bftheta\in L^\infty([0,T],\Theta)}\Big[ \partial_t v(t,\mu) + \langle \partial_\mu v(t,\mu)(.)\cdot \bar{f}(.,\theta) + \bar{L}(.,\theta_s), \,\mu \rangle\Big].
\end{align*}
Passing to the limit $\delta t\rightarrow 0$, we obtain the following HJB equation
\begin{equation}
\label{eq:HJB_dist}
\begin{cases}
\displaystyle{\frac{\partial v}{\partial t} + \inf_{\theta\in\Theta}\left\langle \partial_\mu v(t,\mu)(.)\cdot \bar{f}(.,\theta)+ \bar{L}(.,\theta),\,\mu\right\rangle  = 0,}&\text{on~~} [0,T) \times \cP_2(\R^{d+l}),\\
\displaystyle{v(T, \mu)=\langle\bar{\Phi}(.),\mu\rangle}, &\text{on~~} \cP_2(\R^{d+l}),
\end{cases}
\end{equation}
which the value function should satisfy.
The rest of this and the next section is to establish the precise link between equation~\eqref{eq:HJB_dist} and the value function~\eqref{eq:value_def}. We now prove a verification result, which essentially says that if we have a smooth enough solution of the HJB equation~\eqref{eq:HJB_dist}, then this solution must be the value function. Moreover, the HJB allows us to identify the optimal control policy. 

\begin{proposition}
\label{prop:HJB_verification}
Let $v$ be a function in $C^{1,1}([0,T]\times \cP_2(\R^{d+l}))$. If $v$ is a solution to~\eqref{eq:HJB_dist} and there exists $\theta^\dagger(t,\mu)$, which is a mapping $(t,\mu) \mapsto \Theta$ attaining the infimum in~\eqref{eq:HJB_dist}, then $v(t,\mu)=v^*(t,\mu)$, and $\theta^\dagger$ is an optimal feedback control policy, i.e. $\bftheta=\bftheta^*$ is a solution of~\eqref{eq:mean_field_ctrl_prob}, where
$\theta^*_t := \theta^\dagger(t,\P_{w^*_t})$ with $\P_{w^*_0}=\mu_0$ and $dw^*_t/dt=\bar{f}(w^*_t,\theta^*_t)$. \end{proposition}

\begin{proof}
  Given any control process $\bftheta$, one can apply formula~\eqref{eq:funcito} between $s=t$ and $s=T$ with explicit $t$ dependence and obtain
  \begin{equation*}
  v(T,\P_{T}^{t,\mu,\bftheta}) = v(t,\mu) + \int_{t}^T\frac{\partial v}{\partial t}(s,\P_{s}^{t,\mu,\bftheta}) + \langle\partial_{\mu}v(s,\P_{s}^{t,\mu,\bftheta})(.)\cdot \bar{f}(.;\theta_s),\,\P_{s}^{t,\mu,\bftheta}\rangle\,ds.
  \end{equation*}
  Equivalently, we have
  \begin{align}
  v(t,\mu) &=  v(T,\P_{T}^{t,\mu,\bftheta}) - \int_{t}^T\frac{\partial v}{\partial t}(s,\P_{s}^{t,\mu,\bftheta}) + \langle\partial_{\mu}v(s,\P_{s}^{t,\mu,\bftheta})(.)\cdot \bar{f}(.;\theta_s),\,\P_{s}^{t,\mu,\bftheta}\rangle\,ds\nonumber \\
  &\leq v(T,\P_{T}^{t,\mu,\bftheta}) + \int_{t}^T \langle \bar{L}(.,\theta_s), \,\P_{s}^{t,\mu,\bftheta} \rangle\,ds \nonumber \\
  &=\langle \bar{\Phi}(.),\,\P_{t}^{T,\mu,\bftheta} \rangle + \int_{t}^T \langle \bar{L}(.,\theta_s), \,\P_{s}^{t,\mu,\bftheta} \rangle\,ds \nonumber \\
  &=J(t,\mu,\bftheta),\nonumber
  \end{align}
  where the first inequality comes from the infimum condition in~\eqref{eq:HJB_dist}. Since the control process is arbitrary, we have
  \begin{equation}
  v(t,\mu)\leq v^*(t,\mu).
  \end{equation}
  Replacing the arbitrary control process with $\bftheta^*$ where $\theta^*_t = \theta^\dagger(t,\P_{s}^{t,\mu,\bftheta^*})$ is given by the optimal feedback control and repeating the above argument, noting that the inequality becomes equality since the infimum is attained, we have
  \begin{equation}
  v(t,\mu)=J(t,\mu,\bftheta^*)\geq v^*(t,\mu).
  \end{equation}
  Therefore we obtain $v(t,\mu)=v^*(t,\mu)$ and $\theta^\dagger$ defines an optimal feedback control policy.
\end{proof}

Prop.~\ref{prop:HJB_verification} is an important statement that links smooth solutions of the HJB equation with solutions of the mean-field optimal control problem, and hence the population minimization problem in deep learning. 
Furthermore, by taking the infimum in~\eqref{eq:HJB_dist}, it allows us to identify an optimal control policy $\theta^\dagger:[0,T]\times \cP_2(\R^{d+l}) \rightarrow \Theta$. This is in general a stronger characterization of the solution of the learning problem. In particular, it is of \emph{feedback}, or \emph{closed-loop} form. On the other hand, an \emph{open-loop} solution can be obtained from the closed-loop control policy by sequentially setting $\theta^*_t = \theta^\dagger(t, \P_{w^*_t})$, where $w^*_t$ is the solution of the feed-forward ODE with $\bftheta=\bftheta^*$ up to time $t$. Note that in usual deep learning, the open-loop type solutions are obtained during training and used in inference. In other words, during inference the trained weights are fixed and are not dependent on the distribution of the inputs encountered. On the other hand, controls obtained from closed-loop control policies are actively adjusted according to the distribution encountered. In this sense, the ability to generate an optimal control policy in the form of state-based feedback is an important feature of the dynamic programming approach. However, we should note there is a price to pay for obtaining such a feedback control: the HJB equation is general difficult to solve numerically. We shall return to this point at the end of Sec.~\ref{sec:HJB_vicosity}. 

The limitation of Prop.~\ref{prop:HJB_verification} is that it assumes the value function $v^*(t,\mu)$ is continuously differentiable, which is often not the case. In order to formulate a complete characterization, we would also like to deduce the statement in the other direction: a solution to~\eqref{eq:mean_field_ctrl_prob} should also solve the PDE~\eqref{eq:HJB_dist} in an appropriate sense. In the next section, we achieve this by giving a more flexible characterization of the value function as the viscosity solution of the HJB equation.


\section{Viscosity solution of HJB equation}
\label{sec:HJB_vicosity}
\subsection{The concept of viscosity solutions}
In general, one cannot expect to have smooth solutions to the HJB equation~\eqref{eq:HJB_dist}. 
Therefore we need to extend the classical concept of PDE solutions to a type of weak solutions.
As in the analysis of classical Hamilton-Jacobi equations, we shall introduce a notion of viscosity solution for the HJB equation in the Wasserstein space of probability measures. The key idea is again the lifting identification between measures and random variables, working in the Hilbert space $L^2(\Omega,\R^{d+l})$, instead of the Wasserstein space $\cP_2(\R^{d+l})$. Then, we can use the tools developed for viscosity solutions in Hilbert spaces. The techniques presented below have been employed in the study of well-posedness for general Hamilton-Jacobi equations in Banach spaces, see e.g.~\cite{crandall1985hamilton,crandall1986hamiltona,crandall1986hamiltonb}.

For convenience, we define the Hamiltonian $\mathcal{H}(\xi,P):L^2(\Omega,\R^{d+l})\times L^2(\Omega,\R^{d+l})\rightarrow \R$ as
\begin{equation}
\mathcal{H}(\xi,P)\coloneqq \inf_{\theta\in\Theta}\E[P\cdot \bar{f}(\xi,\theta)+\bar{L}(\xi,\theta)].
\label{eq:lifted_H}
\end{equation}
Then the ``lifted'' Bellman equation of~\eqref{eq:HJB_dist} with $V(t,\xi) = v(t,\P_\xi)$ can be written down as follows, except that the state space is enlarged to $L^2(\Omega,\R^{d+l})$:
\begin{equation}
\label{eq:HJB_lifted}
\begin{cases}
\displaystyle{\frac{\partial V}{\partial t} + \mathcal{H}(\xi,DV(t,\xi)) = 0,}&\text{on~~} [0,T) \times L^2(\Omega,\R^{d+l}),\\
\displaystyle{V(T, \xi)=\E[\bar{\Phi}(\xi)]}, &\text{on~~} L^2(\Omega,\R^{d+l}).
\end{cases}
\end{equation}

\begin{definition}
We say that a bounded, uniformly continuous function $u:\,[0,T]\times\cP_2(\R^{d+l})\rightarrow\R$ is a viscosity (sub, super) solution to~\eqref{eq:HJB_dist} if the lifted function $U:\,[0,T]\times L^2(\Omega,\R^{d+l})\rightarrow\R$ defined by
\begin{equation*}
U(t,\xi)=u(t,\P_\xi)
\end{equation*}
is a viscosity (sub, super) solution to the lifted Bellman equation~\eqref{eq:HJB_lifted}, that is: \\
(i) $U(T,\xi)\leq \E[\bar{\Phi}(\xi)]$, and for any test function $\psi\in C^{1,1}([0,T]\times L^2(\Omega,\R^{d+l}))$ such that the map $U-\psi$ has a local maximum at $(t_0,\xi_0)\in [0,T)\times L^2(\Omega,\R^{d+l})$, one has
\begin{equation}
\label{eq:sub_property}
\partial_t \psi(t_0,\xi_0) + \cH(\xi_0,D\psi(t_0,\xi_0)) \geq 0.
\end{equation}
(ii) $U(T,\xi)\geq \E[\bar{\Phi}(\xi)]$, and for any test function $\psi\in C^{1,1}([0,T]\times L^2(\Omega,\R^{d+l}))$ such that the map $U-\psi$ has a local minimum at $(t_0,\xi_0)\in [0,T)\times L^2(\Omega,\R^{d+l})$, one has
\begin{equation}
\label{eq:super_property}
\partial_t \psi(t_0,\xi_0) + \cH(\xi_0,D\psi(t_0,\xi_0)) \leq 0.
\end{equation}
\end{definition}

\subsection{Existence and uniqueness of viscosity solution}
The main goal of introducing the concept of viscosity solutions is that in the viscosity sense, the HJB equation is well-posed and the value function is the unique solution of the HJB equation. We show this in Thm.~\ref{thm:value_viscosity} and~\ref{thm:value_uniqueness}. 
    
\begin{theorem}
\label{thm:value_viscosity}
The value function $v^*(t,\mu)$ defined in~\eqref{eq:value_def} is a viscosity solution to the HJB equation~\eqref{eq:HJB_dist}.
\end{theorem}

Before proving Thm.~\ref{thm:value_viscosity}, we first introduce a useful Lemma regarding the continuity of $\mathcal{H}(\xi,P)$.
\begin{lemma}
  \label{lem:H_continuity}
  The Hamiltonian $\mathcal{H}(\xi,P)$ defined in~\eqref{eq:lifted_H} satisfies the following continuity conditions:
  \begin{align}
    &|\mathcal{H}(\xi,P)-\mathcal{H}(\xi,Q)|\leq K_L \|P-Q\|_{L^2}, \label{eq:H_continuity1}\\
    &|\mathcal{H}(\xi,P)-\mathcal{H}(\zeta,P)|\leq K_L (1+\|P\|_{L^2})\|\xi-\zeta\|_{L^2}. \label{eq:H_continuity2}
  \end{align}
\end{lemma}

\begin{proof}
  For simplicity we define 
  \begin{equation*}
    \hat{\mathcal{H}}(\xi,P;\theta)\coloneqq \E[P\cdot \bar{f}(\xi,\theta)+\bar{L}(\xi,\theta)].
  \end{equation*}
  The boundedness of $\bar{f}$ and $\bar{L}$ gives us
  \begin{align}
    &|\hat{\mathcal{H}}(\xi,P;\theta) - \hat{\mathcal{H}}(\xi,Q;\theta)| \leq K_L\|P-Q\|_{L^2} \label{eq:H_continuity3} \\
    &|\hat{\mathcal{H}}(\xi,P;\theta) - \hat{\mathcal{H}}(\zeta,P;\theta)| \leq K_L(1+\|P\|_{L^2})\|\xi-\zeta\|_{L^2}. \label{eq:H_continuity4}
  \end{align}
  By definition we know 
  \begin{equation*}
    \mathcal{H}(\xi,P) \coloneqq \inf_{\theta\in\Theta} \hat{\mathcal{H}}(\xi,P;\theta).
  \end{equation*}

  Let $\theta_n$ satisfy
  \begin{equation*}
    \hat{\mathcal{H}}(\xi,P;\theta_n) - \mathcal{H}(\xi,Q) \leq 1/n.
  \end{equation*}
  Then
  \begin{align*}
    &\mathcal{H}(\xi,P)-\mathcal{H}(\xi,Q) \\
    =&\, (\mathcal{H}(\xi,P)- \hat{\mathcal{H}}(\xi,P;\theta_n))
    + (\hat{\mathcal{H}}(\xi,P;\theta_n) - \hat{\mathcal{H}}(\xi,Q;\theta_n)) +
    (\hat{\mathcal{H}}(\xi,Q;\theta_n) - \mathcal{H}(\xi,Q)) \\
    \leq\, & |\hat{\mathcal{H}}(\xi,P;\theta_n) - \hat{\mathcal{H}}(\xi,Q;\theta_n)| + 1/n \\
    \leq\, & K_L\|P-Q\|_{L^2} + 1/n.
  \end{align*}
  Taking $n\rightarrow \infty$, we have $\mathcal{H}(\xi,P)-\mathcal{H}(\xi,Q) \leq K_L\|P-Q\|_{L^2}$. A similar computation shows $\mathcal{H}(\xi,Q)-\mathcal{H}(\xi,P) \leq K_L\|P-Q\|_{L^2}$, and we prove~\eqref{eq:H_continuity1}.
 ~\eqref{eq:H_continuity2} can be proved in a similar way, based on the condition~\eqref{eq:H_continuity4}.
\end{proof}

\begin{proof}[Proof of Thm.~\ref{thm:value_viscosity}]
  We lift the value function $v^*(t,\mu)$ to $[0,T]\times L^2(\Omega,\R^{d+l})$ and denote it by $V^*(t,\xi)$. Note that the convergence $\xi_n\rightarrow \xi$ in $L^2(\Omega,\R^{d+l})$ implies the convergence $\P_{\xi_n}\rightarrow \P_{\xi}$ in $\cP_2(\R^{d+l})$, thus Prop. \ref{prop:value_continuity} guarantees that $V^*(t,\xi)$ is continuous on $[0,T]\times L^2(\Omega,\R^{d+l})$. By definition we know $V^*(t,\xi)$ is bounded and $V^*(T,\xi)=\E(\bar{\Phi}(\xi))$. It remains to show the viscosity sub and super solution properties of $V^*(t,\xi)$. To proceed, we note that $V^*(t,\xi)$ also inherits the dynamic programming principle from $v^*(t,\mu)$ (c.f.\,Prop.~\ref{prop:DPP}), which can be represented as
  \begin{equation}
    \label{eq:DPP_V}
    V^*(t,\xi) = \inf_{\bftheta\in L^\infty([0,T],\Theta)}\Big[ \int_{t}^{\hat{t}}\E [\bar{L}(W_s^{t,\xi,\bftheta},\theta_s)]\,ds + V^*(\hat{t}, W_{\hat{t}}^{t,\xi,\bftheta})\Big].
  \end{equation}

  \textit{1. Subsolution property.} Suppose $\psi$ is a test function in $C^{1,1}([0,T]\times L^2(\Omega,\R^{d+l}))$ and $V^*-\psi$ has a local maximum at $(t_0,\xi_0)\in [0,T)\times L^2(\Omega,\R^{d+l})$, which means 
  \begin{equation*}
    (V^*-\psi)(t,\xi)\leq (V^*-\psi)(t_0,\xi_0) \text{ for all } (t,\xi) \text{ satisfying } |t-t_0|+\|\xi-\xi_0\|_{L^2} < \delta.
  \end{equation*}
  Let $\theta_0$ be an arbitrary element in $\Theta$ and define a control process $\bftheta^0 \in L^\infty([0,T],\Theta)$ such that $\theta^0_s \equiv \theta_0,\,s\in[t_0,T]$. Let $h\in (0,T-t_0)$ be small enough such that $|s-t_0|+\|W_s^{t_0,\xi_0,\bftheta^0}-\xi_0\|_{L^2} < \delta$ for all $s\in[t_0,t_0+h]$. This is possible from an argument similar in the proof of Prop. \ref{prop:value_continuity}. From the dynamic programming principle~\eqref{eq:DPP_V}, we have
  \begin{equation*}
    V^*(t_0,\xi_0) \leq \int_{t_0}^{t_0+h}\E [\bar{L}(W_s^{t_0,\xi_0,\bftheta^0},\theta^0_s)]\,ds + V^*(t_0+h, W_{t_0+h}^{t_0,\xi_0,\bftheta^0}) .
  \end{equation*}
  Using the condition of local maximality and chain rule~\eqref{eq:funcito_lift}, we have the inequality
  \begin{align}
    0 &\leq V^*(t_0+h, W_{t_0+h}^{t_0,\xi_0,\bftheta^0}) - V^*(t_0,\xi_0) +  \int_{t_0}^{t_0+h}\E [\bar{L}(W_s^{t_0,\xi_0,\bftheta^0},\theta^0_s)]\,ds \nonumber \\
    &\leq \psi(t_0+h, W_{t_0+h}^{t_0,\xi_0,\bftheta^0}) - \psi(t_0,\xi_0) +  \int_{t_0}^{t_0+h}\E [\bar{L}(W_s^{t_0,\xi_0,\bftheta^0},\theta^0_s)]\,ds \nonumber \\
    &=  \int_{t_0}^{t_0+h} \partial_t\psi(s,W_s^{t_0,\xi_0,\bftheta^0}) + 
    \E [D\psi(s,W_s^{t_0,\xi_0,\bftheta^0})\cdot \bar{f}(W_s^{t_0,\xi_0,\bftheta^0},\theta^0_s)]\,ds \nonumber \\
    &\hphantom{~=} + \int_{t_0}^{t_0+h}\E [\bar{L}(W_s^{t_0,\xi_0,\bftheta^0},\theta^0_s)]\,ds. \label{eq:sub_inequal1}
  \end{align}
  Since we know $W_s^{t_0,\xi_0,\bftheta^0}$ is continuous in time, in the sense of $L^2$-metric of $L^2(\Omega,\R^{d+l})$, hence 
  \begin{align*}
    \partial_t\psi(s,W_s^{t_0,\xi_0,\bftheta^0}) + 
    \E [D\psi(s,W_s^{t_0,\xi_0,\bftheta^0})\cdot \bar{f}(W_s^{t_0,\xi_0,\bftheta^0},\theta^0_s) + \bar{L}(W_s^{t_0,\xi_0,\bftheta^0},\theta^0_s)]
  \end{align*}
  is also continuous in time. Dividing the inequality~\eqref{eq:sub_inequal1} by $h$ and taking the limit $h\rightarrow 0$, we obtain
  \begin{align*}
    0 &\leq \Big[\partial_t\psi(s,W_s^{t_0,\xi_0,\bftheta^0}) + 
    \E [D\psi(s,W_s^{t_0,\xi_0,\bftheta^0})\cdot \bar{f}(W_s^{t_0,\xi_0,\bftheta^0},\theta^0_s) + \bar{L}(W_s^{t_0,\xi_0,\bftheta^0},\theta^0_s)]\Big]\Big |_{s=t_0} \\
    &= \partial_t\psi(t_0,\xi_0) + 
    \E [D\psi(t_0,\xi_0)\cdot \bar{f}(\xi_0,\theta_0) + \bar{L}(\xi_0,\theta_0)].
  \end{align*}
  Since $\theta_0$ is arbitrary in $\Theta$, we obtain the desired subsolution property~\eqref{eq:sub_property}.

  \textit{2. Supersolution property.} Suppose $\psi$ is a test function in $C^{1,1}([0,T]\times L^2(\Omega,\R^{d+l}))$ and $V^*-\psi$ has a local minimum at $(t_0,\xi_0)\in [0,T)\times L^2(\Omega,\R^{d+l})$, which means 
  \begin{equation*}
    (V^*-\psi)(t,\xi)\geq (V^*-\psi)(t_0,\xi_0) \text{ for all } (t,\xi) \text{ satisfying } |t-t_0|+\|\xi-\xi_0\|_{L^2} < \delta_1.
  \end{equation*}
  Given an arbitrary $\veps>0$, since Lemma \ref{lem:H_continuity} tells us $\mathcal{H}$ is continuous, there exits $\delta_2 >0$ such that 
  \begin{eqnarray*}
    |\partial_t\psi(t,\xi)+\mathcal{H}(t,\xi) - \partial_t\psi(t_0,\xi_0)  - \mathcal{H}(t_0,\xi_0)|<\veps,
  \end{eqnarray*}
  for all $(t,\xi) \text{ satisfying } |t-t_0|+\|\xi-\xi_0\|_{L^2} < \delta_2$.
  Again as argued in the proof of Prop. \ref{prop:value_continuity}, we can choose $h\in (0,T-t_0)$ to be small enough such that $|s-t_0|+\|W_s^{t_0,\xi_0,\bftheta}-\xi_0\|_{L^2} < \min\{\delta_1,\delta_2\}$ for all $s\in[t_0,t_0+h]$, $\bftheta \in L^\infty([0,T],\Theta)$.

  From the dynamic programming principle~\eqref{eq:DPP_V}, there exists $\bftheta^h$ such that
  \begin{equation*}
    V^*(t_0,\xi_0) + \veps h \geq \int_{t_0}^{t_0+h}\E [\bar{L}(W_s^{t_0,\xi_0,\bftheta^h},\theta^h_s)]\,ds + V^*(t_0+h, W_{t_0+h}^{t_0,\xi_0,\bftheta^h}) .
  \end{equation*}
  Again using the condition of local minimality, chain rule~\eqref{eq:funcito_lift}, and definition of $\mathcal{H}$, we have the inequality
  \begin{align}
    \veps h &\geq V^*(t_0+h, W_{t_0+h}^{t_0,\xi_0,\bftheta^h}) - V^*(t_0,\xi_0) +  \int_{t_0}^{t_0+h}\E [\bar{L}(W_s^{t_0,\xi_0,\bftheta^h},\theta^h_s)]\,ds \nonumber \\
    &\geq \psi(t_0+h, W_{t_0+h}^{t_0,\xi_0,\bftheta^h}) - \psi(t_0,\xi_0) +  \int_{t_0}^{t_0+h}\E [\bar{L}(W_s^{t_0,\xi_0,\bftheta^h},\theta^h_s)]\,ds \nonumber \\
    &=  \int_{t_0}^{t_0+h} \partial_t\psi(s,W_s^{t_0,\xi_0,\bftheta^h}) + 
    \E [D\psi(s,W_s^{t_0,\xi_0,\bftheta^h})\cdot \bar{f}(W_s^{t_0,\xi_0,\bftheta^h},\theta^h_s)]\,ds \nonumber  \\
    &\hphantom{~=} + \int_{t_0}^{t_0+h}\E [\bar{L}(W_s^{t_0,\xi_0,\bftheta^h},\theta^h_s)]\,ds \nonumber \\
    &\geq  \int_{t_0}^{t_0+h} \partial_t\psi(s,W_s^{t_0,\xi_0,\bftheta^h}) + \mathcal{H}(W_s^{t_0,\xi_0,\bftheta^h}, D\psi(s,W_s^{t_0,\xi_0,\bftheta^h}))\,ds \nonumber \\
    &\geq h (\partial_t\psi(t_0,\xi_0) + \mathcal{H}(t_0,\xi_0) - \veps). \label{eq:super_inequal1}
  \end{align}
  Dividing the inequality~\eqref{eq:super_inequal1} by $h$ and taking the limit $\veps \rightarrow 0$, we obtain the desired supersolution property~\eqref{eq:super_property}.
\end{proof}

Theorem~\ref{thm:value_viscosity} incidentally establishes the existence of viscosity solutions to the HJB, which we can identify as the value function of the mean-field control problem. We show below that this solution is in fact unique. 

\begin{theorem}
\label{thm:value_uniqueness}
Let $u_1$ and $u_2$ be two functions defined on $[0,T]\times\cP_2(\R^{d+l})$ such that $u_1$ and $u_2$ are viscosity subsolution and supersolution to~\eqref{eq:HJB_dist} respectively. Then $u_1\leq u_2$. Consequently, the value function $v^*(t,\mu)$ defined in~\eqref{eq:value_def} is the unique viscosity solution to the HJB equation~\eqref{eq:HJB_dist}.
\end{theorem}

\begin{proof}
  The final assertion of the theorem follows immediately from Thm. \ref{thm:value_viscosity}.  As before we consider the lifted version $U_1(t,\xi)=u_1(t,\P_{\xi})$, $U_2(t,\xi)=u_2(t,\P_{\xi})$ on $[0,T]\times L^2(\Omega,\R^{d+l})$. By definition we know $U_1$ and $U_2$ are subsolution and supersolution to~\eqref{eq:HJB_lifted} respectively. By definition of viscosity solution, $U_1,U_2$ are both bounded and uniformly continuous. We denote their moduli of continuity by $\omega_1,\omega_2$, which satisfy
  \begin{equation*}
    |U_i(t,\xi)-U_i(s,\zeta)|\leq \omega_i(|t-s|+\|\xi-\zeta\|_{L^2}), \quad i=1,2
  \end{equation*}
  for all $0\leq t\leq s\leq T, \xi,\zeta\in L^2(\Omega,\R^{d+l})$, and $\omega_i(r)\rightarrow0$ as $r\rightarrow 0^{+}$.
  To prove $U_1\leq U_2$, we assume
  \begin{equation}
    \delta\coloneqq \sup_{[0,T]\times L^2(\Omega,\R^{d+l})} U_1(t,\xi)-U_2(t,\xi)> 0,
  \end{equation}
  and proceed in five steps below to derive a contradiction.

  1). Let $\sigma,\veps\in(0,1)$ and construct the auxiliary function
  \begin{equation}
    G(t,s,\xi,\zeta)=U_1(t,\xi)-U_2(s,\zeta) + \sigma(t+s)-\veps(\|\xi\|^2_2+\|\zeta\|^2_2)-\frac{1}{\veps^2}((t-s)^2+\|\xi-\zeta\|_{L^2}^2),
  \end{equation}
  for $t,s\in[0,T], \xi,\zeta\in L^2(\Omega,\R^{d+l})$.
  From Stegall Theorem~\cite{stegall1978optimization} there exist $\eta_t,\eta_s\in\R$, $\eta_{\xi},\eta_{\zeta}\in L^2(\Omega,\R^{d+l})$ such that $|\eta_t|,|\eta_s|,\|\eta_{\xi}\|_{L^2},\|\eta_{\zeta}\|_{L^2}\leq \veps$ and the function with linear perturbation
  \begin{equation}
    \tilde{G}(t,s,\xi,\zeta)\coloneqq G(t,s,\xi,\zeta)-\eta_t t -\eta_s s-\E[\eta_\xi\cdot \xi] - \E[\eta_{\zeta}\cdot \zeta]
  \end{equation}
  has a maximum over $[0,T]\times [0,T] \times L^2(\Omega,\R^{d+l}) \times L^2(\Omega,\R^{d+l})$ at $(t_0,s_0,\xi_0,\zeta_0)$. 

  2). Since $\tilde{G}(0,0,0,0) \leq \tilde{G}(t_0,s_0,\xi_0,\zeta_0)$ and $U_1,U_2$ are bounded, after an arrangement of terms, we have
  \begin{align}
    &\veps(\|\xi_0\|_{L^2}^2 + \|\zeta_0\|_{L^2}^2) \nonumber \\
    \leq\,& C + \sigma (t_0 + s_0) -\frac{1}{\veps^2}((t_0-s_0)^2+\|\xi_0-\zeta_0\|_{L^2}^2) -\eta_t t_0 - \eta_s s_0 \nonumber \\
    & - \E[\eta_\xi\cdot \xi_0] - \E[\eta_{\zeta}\cdot \zeta_0]  \nonumber \\
    \leq\,& C - \E[\eta_\xi\cdot \xi_0] - \E[\eta_{\zeta}\cdot \zeta_0] \nonumber \\
    \leq\,& C + \sqrt{2}\,\veps(\|\xi_0\|_{L^2}^2 + \|\zeta_0\|_{L^2}^2)^{1/2}. \label{eq:unique_ineq1}
  \end{align}
  Here and in the following $C$ denotes generic positive constant, whose value may change from line to line but is always independent of $\veps$ and $\sigma$. Solving the quadratic inequality above, we get
  \begin{equation}
    (\|\xi_0\|_{L^2}^2 + \|\zeta_0\|_{L^2}^2)^{1/2} \leq C(1+\veps^{-1/2}).
    \label{eq:unique_est1}
  \end{equation}
  Now arguing in the same way as~\eqref{eq:unique_ineq1} and further combining~\eqref{eq:unique_ineq1}, we have
  \begin{align*}
    \frac{1}{\veps^2}((t_0-s_0)^2+\|\xi_0-\zeta_0\|_{L^2}^2) 
    & \leq  C - \E[\eta_\xi\cdot \xi_0] - \E[\eta_{\zeta}\cdot \zeta_0] \\
    & \leq  C + \sqrt{2}\,\veps(\|\xi_0\|_{L^2}^2 + \|\zeta_0\|_{L^2}^2)^{1/2} \\
    & \leq C,
  \end{align*}
  or equivalently 
  \begin{equation}
    |t_0-s_0| + \|\xi_0-\zeta_0\|_{L^2} \leq C\veps.
    \label{eq:unique_est2}
  \end{equation}
  
  3). Eq.~\eqref{eq:unique_est2} allows us to further sharpen the estimate of $(t-s)^2+\|\xi-\zeta\|_{L^2}^2$. Specifically, since $\tilde{G}(t_0,t_0,\xi_0,\xi_0) \leq \tilde{G}(t_0,s_0,\xi_0,\zeta_0)$, we have
  \begin{align*}
    & \E[\eta_{s}\cdot (s_0-t_0)] + \E[\eta_{\zeta}\cdot (\zeta_0-\xi_0)] \\
    \leq\, & U_2(t_0,\xi_0)-U_2(s_0,\zeta_0) + \sigma(s_0-t_0) + \veps(\|\xi_0\|_{L^2}^2 - \|\zeta_0\|_{L^2}^2)\\
    & - \frac{1}{\veps^2}((t_0-s_0)^2+\|\xi_0-\zeta_0\|_{L^2}^2). 
  \end{align*}

  Rearranging the above inequality and using estimates~\eqref{eq:unique_est1},~\eqref{eq:unique_est2}, and uniform continuity of $U_2$, we obtain
  \begin{align*}
    &\frac{1}{\veps^2}((t_0-s_0)^2+\|\xi_0-\zeta_0\|_{L^2}^2)  \\
    \leq\, & \omega_2(|t_0-s_0|+\|\xi_0-\zeta_0\|_{L^2}) + C(|t_0-s_0|+\|\xi_0-\zeta_0\|_{L^2}) + 
    \veps \|\xi_0+\zeta_0\|_{L^2} \|\xi_0 - \zeta_0\|_{L^2} \\
    \leq\, & \omega_2(|t_0-s_0|+\|\xi_0-\zeta_0\|_{L^2}) + C(|t_0-s_0|+\|\xi_0-\zeta_0\|_{L^2}) \\
    \leq\, & \omega_2(C\veps) + C\veps.
  \end{align*}
  By the property of modulus, we conclude
  \begin{equation}
    |t_0-s_0| + \|\xi_0-\zeta_0\|_{L^2} = o(\veps).
    \label{eq:unique_est3}
  \end{equation}

  4). From the definition of $\tilde{G}$ and $\delta$, we can choose $\veps$ so small  that 
  \begin{equation*}
    \sup_{[0,T]\times L^2(\Omega,\R^{d+l})}\tilde{G}(t,t,\xi,\xi)\geq \frac{\delta}{2}.
  \end{equation*}
  Using estimate~\eqref{eq:unique_est1},~\eqref{eq:unique_est3}, we can furthermore choose $\sigma, \veps$ small enough such that 
  \begin{align*} 
    & U_1(t_0,\xi_0)-U_2(s_0,\zeta_0) \\
    \geq\, &  \tilde{G}(t_0,s_0,\xi_0,\zeta_0) -  C\sigma - C\veps \\
    \geq\, &  \sup_{[0,T]\times L^2(\Omega,\R^{d+l})}\tilde{G}(t,t,\xi,\xi) - \frac{\delta}{4} \\
    \geq\, & \frac{\delta}{4}.
  \end{align*}
  Noting the terminal condition $U_1(T,\xi)\leq U_2(T,\xi)$, we are ready to estimate $|T-t_0|$ through
  \begin{align*}
    \frac{\delta}{4} \leq\, & U_1(t_0,\xi_0)-U_2(s_0,\zeta_0) \\
    \leq\,& U_1(t_0,\xi_0)-U_1(T,\xi_0) + U_1(T,\xi_0)-U_2(T,\xi_0) \\
    &+ U_2(T,\xi_0)-U_2(t_0,\xi_0) + 
    U_2(t_0,\xi_0)-U_2(s_0,\zeta_0) \\
    \leq\,& \omega_1(|T-t_0|) + \omega_2(|T-t_0|) + \omega_2(|t_0-s_0| + \|\xi_0-\zeta_0\|_{L^2}) \\
    =\,& \omega_1(|T-t_0|) + \omega_2(|T-t_0|) + \omega_2(o(\veps)).
  \end{align*}
  Therefore, when $\veps$ is small enough, we have
  \begin{equation*}
    \omega_1(|T-t_0|) + \omega_2(|T-t_0|) \geq \frac\delta8,
  \end{equation*}
  which implies
  \begin{equation*}
    |T-t_0|\geq \lambda > 0,
  \end{equation*}
  for some positive constant $\lambda$, provided $\sigma,\veps$ are small enough. The same argument as above can also give $|T-s_0|\geq \lambda > 0$.

  5). The finite differences between $t_0,s_0$ and $T$ finally allow us to employ the viscosity property. Consider the map $(t,\xi)\mapsto \tilde{G}(t,s_0,\xi,\zeta_0)$ has a maximum at $(t_0,\xi_0)$, i.e. $U_1-\psi$ has a maximum at $(t_0,\xi_0)$ for
  \begin{align*}
    \psi(t,\xi)\coloneqq &\,U_2(s_0,\zeta_0) - \sigma (t + s_0) + \veps(\|\xi\|_{L^2}^2+\|\zeta_0\|_{L^2}^2)+\frac{1}{\veps^2}((t-s_0)^2+\|\xi-\zeta_0\|_{L^2}^2) \\
    & + \eta_t t + \eta_s s_0 + \E[\eta_\xi\cdot \xi] + \E[\eta_{\zeta}\cdot \zeta_0].
  \end{align*}
  Since $U_1$ is a viscosity subsolution, using the subsolution property~\eqref{eq:sub_property}, we have
  \begin{equation}
    -\sigma + \frac{2(t-s_0)}{\veps^2} + \eta_t  + 
    \mathcal{H}(\xi_0,2\veps\xi_0 + \frac{2(\xi_0-\zeta_0)}{\veps^2} + \eta_{\xi})\geq 0.
    \label{eq:unique_H_ineq1}
  \end{equation}
  In the same way, consider the map $(s,\zeta)\mapsto -\tilde{G}(t_0,s,\xi_0,\zeta)$ has a minimum at $(s_0,\zeta_0)$, i.e. $U_2-\psi$ has a minimum at $(s_0,\zeta_0)$ for
  \begin{align*}
    \psi(t,\xi)\coloneqq &\,U_1(s_0,\zeta_0) + \sigma (t_0 + s) - \veps(\|\xi_0\|_{L^2}^2+\|\zeta\|_{L^2}^2) -
    \frac{1}{\veps^2}((t_0-s)^2+\|\xi_0-\zeta\|_{L^2}^2) \\
    & - \eta_t t_0 - \eta_s s - \E[\eta_\xi\cdot \xi_0] - \E[\eta_{\zeta}\cdot \zeta].
  \end{align*}
  Since $U_2$ is a viscosity supersolution, using the supersolution property~\eqref{eq:super_property}, we have
  \begin{equation}
    \sigma + \frac{2(t_0-s)}{\veps^2} - \eta_s + 
    \mathcal{H}(\zeta_0,-2\veps\zeta_0 + \frac{2(\xi_0-\zeta_0)}{\veps^2} - \eta_{\zeta})\geq 0.
    \label{eq:unique_H_ineq2}
  \end{equation}
  Computing the difference in the two inequalities~\eqref{eq:unique_H_ineq1},\eqref{eq:unique_H_ineq2} gives
  \begin{equation*}
    -2\sigma + \eta_t + \eta_s + \mathcal{H}(\xi_0,2\veps\xi_0 + \frac{2(\xi_0-\zeta_0)}{\veps^2} + \eta_{\xi}) - \mathcal{H}(\zeta_0,-2\veps\zeta_0 + \frac{2(\xi_0-\zeta_0)}{\veps^2} - \eta_{\zeta}) \geq 0.
  \end{equation*}
  Using estimates~\eqref{eq:unique_est1},~\eqref{eq:unique_est3} and Lemma \ref{lem:H_continuity}, we have
  \begin{align*}
    2\sigma \leq\,& \eta_t + \eta_s
    + \mathcal{H}(\zeta_0, -2\veps\zeta_0 + \frac{2(\xi_0-\zeta_0)}{\veps^2} - \eta_{\zeta})
    - \mathcal{H}(\xi_0, 2\veps\xi_0 + \frac{2(\xi_0-\zeta_0)}{\veps^2} + \eta_{\xi}) \\
    \leq\,& 2\veps 
    + |\mathcal{H}(\zeta_0, -2\veps\zeta_0 + \frac{2(\xi_0-\zeta_0)}{\veps^2} - \eta_{\zeta})
    - \mathcal{H}(\zeta_0, 2\veps\xi_0 + \frac{2(\xi_0-\zeta_0)}{\veps^2} + \eta_{\xi})| \\
    & + |\mathcal{H}(\zeta_0, 2\veps\xi_0 + \frac{2(\xi_0-\zeta_0)}{\veps^2} + \eta_{\xi})
    - \mathcal{H}(\xi_0,2\veps\xi_0 + \frac{2(\xi_0-\zeta_0)}{\veps^2} + \eta_{\xi})| \\
    \leq\,& 2\veps + K_L \| 2\veps\xi_0 + 2\veps\zeta_0 + \eta_{\xi} + \eta_{\zeta} \|_{L^2} \\
    & + K_L(1 + \| 2\veps\xi_0 + \frac{2(\xi_0-\zeta_0)}{\veps^2} + \eta_{\xi}\|_{L^2}) \| \xi_0 - \zeta_0 \|_{L^2} \\
    \leq\,& o(1) \quad (\veps \rightarrow 0^+).
  \end{align*}
  Therefore taking the limit gives us a contradiction $0 < \sigma \leq 0$, which completes the proof.
\end{proof}

Thm.~\ref{thm:value_viscosity} and~\ref{thm:value_uniqueness} establishes the well-posedness, in the viscosity sense, of the HJB equation and identifies the value function for the mean-field optimal control problem as the unique solution of the HJB equation. Moreover, it provides us (through solving the infimum in~\eqref{eq:HJB_dist} after solving for the value function) an optimal control policy, from which we can synthesize an optimal control as the solution of our learning problem. In this sense, the HJB equation gives us a necessary and sufficient condition for optimality of the learning problem~\eqref{eq:mean_field_ctrl_prob}. This demonstrate an essential observation from the mean-field optimal control viewpoint of deep learning: the population risk minimization problem of deep learning can be viewed as a variational problem, whose solution can be characterized by a suitably defined Hamilton-Jacobi-Bellman equation. This very much parallels classical calculus of variations. 

It is worth noting that the HJB equation is a global characterization of the value function, in the sense that it must in principle be solved over the entire space $\cP_2(\R^{d+l})$ of input-target distributions. Of course, we would not expect this to be the case in practice for any non-trivial machine learning problem. However, if we can solve it locally around some trajectories generated by the initial condition $\mu_0 \in \cP_2(\R^{d+l})$, then we would expect the obtained feedback control policy to apply to nearby input-label distributions as well. This may be able to give a principled way to perform transfer or one-shot learning~\cite{pan2010survey,glorot2011domain,li2006one}. 

Finally, observe that if the Hamiltonian defined in~\eqref{eq:lifted_H} is attained by a unique minimizer $\theta^*\in \Theta$ given any $\xi \in L^2(\Omega,\R^{d+l})$ and $P\in L^2(\Omega,\R^{d+l})$, then the uniqueness of value function immediately implies the uniqueness of the open-loop optimal control, which is sometimes a desired property of the population risk minimization problem. The following example gives such an instance.

\begin{example}
\label{ex:linear_with_theta}
  Consider a specific type of residual networks, where $f(x,\theta) = \theta \sigma(x)$ and $L(x,\theta)\propto \Vert \theta \Vert^2$. Here $\theta\in\R^{d\times d}$ is a matrix and $\sigma$ is a smooth and bounded non-linearity, e.g.,\,tanh or sigmoid. This is similar to conventional residual neural networks except that the order of the affine transformation and the non-linearity are swapped. In this case, the Hamiltonian defined in~\eqref{eq:lifted_H} admits a unique minimizer $\theta^*$ given any $\xi\in L^2(\Omega,\R^{d+l})$ and $P\in L^2(\Omega,\R^{d+l})$. 
\end{example}


\section{Mean-field Pontryagin's maximum principle}
\label{sec:mean_field_pmp}

As discussed in the earlier sections, the HJB equation provides us with a complete characterization of the optimality conditions for the population risk minimization problem~\eqref{eq:mean_field_ctrl_prob}. However, it has the disadvantage that it is global in $\cP(\R^{d+l})$ (or its lifted version, in $L^2(\Omega, \R^{d+l})$, and hence difficult to handle in practice. The natural question is whether we can have a local characterization of optimality, and by local we mean having no need for the optimality condition to depend on the whole space of input-label distributions. In this section, we provide such a characterization by proving a mean-field version of the celebrated Pontryagin's maximum principle (PMP)~\cite{boltyanskii1960theory}. Although seemingly disparate at first, we will discuss in Subsec.~\ref{sec:principle_connection} that the maximum principle approach is intimately connected with the dynamic programming approach introduced earlier.

In classical optimal control, such a local characterization is given in the form of the Pontryagin's maximum principle, where forward and backward Hamiltonian dynamics are coupled through a maximization condition. In the present formulation, a common control parameter is shared by all input-target pair values $(x_0,y_0)$ that can take under the distribution $\mu_0$. Thus, one expects that a maximum principle should exist in the average sense. Let us state and prove such a maximum principle below. We modify the assumptions (A1), (A2) to
\begin{itemize}
    \item[(A1$'$)] The function $f$ is bounded; $f,L$ are continuous in $\theta$; and $f,L,\Phi$ are continuously differentiable with respect to $x$.
    \item[(A2$'$)] The distribution $\mu_0$ has bounded support in $\R^d\times \R^l$, i.e.\,there exists $M>0$ such that $\mu (\{ (x,y) \in \R^d \times \R^l : \Vert x \Vert + \Vert y \Vert \leq M \}) = 1$. 
\end{itemize}

\begin{theorem}[Mean-field PMP]
\label{thm:mean_field_PMP}
	Let (A1$'$), (A2$'$) be satisfied and $\bftheta^*\in L^\infty([0,T],\Theta)$ be a solution of~\eqref{eq:mean_field_ctrl_prob} in the sense that $J(\bftheta^*)$ attains the infimum. Then, there exists absolutely continuous stochastic processes $\bfx^*,\bfp^*$ such that
	\begin{align}
        & \dot{x}^*_t = f(x^*_t, \theta^*_t),
        & & x^*_t = x_0, \label{eq:pmp_state} \\ 
        & \dot{p}^*_t = - \nabla_x H(x^*_t, p^*_t, \theta^*_t),
        & & p^*_T = -\nabla_x \Phi(x^*_T, y_0), \label{eq:pmp_costate} \\
		& \E_{\mu_0} H(x^*_t, p^*_t, \theta^*_t)
        \geq \E_{\mu_0} H(x^*_t, p^*_t, \theta),
        & &
        \forall\,\theta \in \Theta, \quad a.e.\,t\in[0,T], \label{eq:pmp_max}
	\end{align}
	where the Hamiltonian function $H: \R^d \times \R^d \times \Theta \rightarrow \R$ is given by 
  \begin{equation}
    \label{eq:pmp_H}
    H(x,p,\theta) = p\cdot f(x, \theta) - L(x, \theta).
  \end{equation}
\end{theorem}

\begin{proof}
  To simplify the proof we first make a substitution by introducing a new coordinate $x^0$ satisfying the dynamics
  $\dot{x}^0_t = L(x_t,\theta_t)$ with $x^0_0=0$. Then, it is clear that the PMP above can be transformed into one without running loss by redefining
  \begin{align*}
      x \rightarrow (x^0, x), \quad
      f \rightarrow (L, f), \quad
      \Phi(x_T, y_0) &\rightarrow \Phi(x_T, y_0) + x^0_T.
  \end{align*}
  Check that (A1$'$), (A2$'$) are preserved but now we can consider without loss of generality the case $L\equiv 0$.

  Let some $\tau \in (0,T]$ be a Lebesgue point of $\hat{f}(t)\coloneqq f(x^*_t,\theta^*_t)$. By assumptions (A1$'$) and (A2$'$)  these points are dense in $[0,T]$. Now, for $\epsilon\in(0,\tau)$, define the family of perturbed controls
  \begin{align*}
      \theta^{\tau,\epsilon}_t = 
      \begin{cases}
          \omega & t\in[\tau-\epsilon, \tau], \\
          \theta^*_t & \text{otherwise}.
      \end{cases}
  \end{align*}
  where $\omega\in\Theta$. This is a ``needle'' perturbation. Accordingly, define $x^{\tau,\epsilon}_t$ by
  \begin{align*}
      x^{\tau,\epsilon}_t = x_0 + \int_{0}^{t} f(x^{\tau,\epsilon}_s, \theta^{\tau,\epsilon}_s) ds.
  \end{align*}
  i.e.\,solution of the forward propagation equation with the perturbed control $\theta^{\tau,\epsilon}$. It is clear that $x^*_t = x^{\tau,\epsilon}_t$ for every $t<\tau-\epsilon$ and every $x_0$, since the perturbation is not present. At $t=\tau$, we have
  \begin{align*}
      \frac{1}{\epsilon} (x^{\tau,\epsilon}_\tau - x^*_\tau)
      &= \frac{1}{\epsilon} \int_{\tau-\epsilon}^{\tau}
      f(x^{\tau,\epsilon}_s,\omega) - f(x^*_s,\theta^*_s) ds.
  \end{align*}
  Since $\tau$ is Lebesgue point of $F$, we have
  \begin{align*}
      v_\tau \coloneqq \lim_{\epsilon\downarrow 0} \frac{1}{\epsilon} (x^{\tau,\epsilon}_\tau - x^*_\tau)
      = f(x^{*}_\tau,\omega) - f(x^*_\tau,\theta^*_\tau).
  \end{align*}
  Here, $v_\tau$ represents the leading order perturbation on the state due to the ``needle'' perturbation introduced in the infinitesimal interval $[\tau-\epsilon,\tau]$. 
  For the rest of the time interval $(\tau,T]$, the dynamics remain the same since the controls are the same. It remains to compute how the perturbation $v_\tau$ propagates. 
  Define for $t\geq\tau$, $v_{t}^\epsilon \coloneqq \frac{1}{\epsilon} (x^{\tau,\epsilon}_t - x^*_t)$
  and $v_t \coloneqq \lim_{\epsilon\downarrow 0} v^\epsilon_t$.
  By Theorem 2.3.1 of~\cite{bressan2007introduction}, we know that $v_t$ is well defined for almost every $t$ (all the Lebesgue points of the map $t\mapsto x^*(t)$) and satisfies the following linearized equation:
  \begin{align}
      \begin{split}
          \dot{v}_t &= \nabla_x f(x^*_t, \theta^*_t)^T v_t, \qquad t\in(\tau,T], \\
          v_\tau &= f(x^*_\tau,\omega) - f(x^*_\tau,\theta^*_\tau).
      \end{split}
      \label{eq:variational_eqn}
  \end{align}
  In particular, $v(T)$ represents the perturbation of the final state introduced by this control. By the optimality assumption of $\bftheta^*$, we must have
  \begin{align*}
      \E_{\mu_0} 
      \Phi(x^{\tau,\epsilon}_T, y_0)
      \geq \E_{\mu_0} \Phi(x^*_T, y_0). 
  \end{align*}
  Assumption (A1$'$) and (A2$'$) implies $\nabla_x \Phi$ is bounded so by dominated convergence theorem,
  \begin{align}
      \begin{split}
          0 &\leq \lim_{\epsilon\downarrow 0} 
          \frac{1}{\epsilon} \E_{\mu_0} 
          [\Phi(x^{\tau,\epsilon}_T,y_0) - \Phi(x^*_T,y_0)] \\
          &= \E_{\mu_0} \frac{d}{d\epsilon} \Phi(x^{\epsilon,\tau}_T, y_0)\Big\vert_{\epsilon=0^+} \\
          &= \E_{\mu_0} \nabla_x \Phi(x^*_T, y_0) \cdot v_T . 
      \end{split}
      \label{eq:final_opt_cond}
  \end{align}

  Now, let us define $\bfp^*$ to be the solution of the adjoint of Eq~\eqref{eq:variational_eqn},
  \begin{align*}
      \dot{p}^*_t = - \nabla_x f(x^*_s, \theta^*_s) p^*_t, \quad p^*_T = -\nabla_x \Phi(x^*_T,y_0).
  \end{align*}
  Then,~\eqref{eq:final_opt_cond} implies $\E_{\mu_0} p^*_T \cdot v_T \leq 0$. Moreover, we have
  \begin{align*}
      \frac{d}{dt} (p^*_t \cdot v_t) = \dot{p}^*_t \cdot v_t + \dot{v}_t \cdot p^*_t = 0
  \end{align*}
  for all $t\in[\tau,T]$. Thus, we must have $\E_{\mu_0} p^*_t\cdot v_t = \E_{\mu_0} p^*_T\cdot v_T \leq 0$ for all $t\in[\tau,T]$ and so for $t=\tau$ (with initial condition in~\eqref{eq:variational_eqn}),
  \begin{align*}
      \E_{\mu_0} p^*_\tau \cdot f(x^*_\tau,\theta^*_\tau)
      \geq 
      \E_{\mu_0} p^*_\tau \cdot f(x^*_\tau,\omega). 
  \end{align*}
  Since $\omega\in\Theta$ is arbitrary, this completes the proof by recalling that $H(x,p,\theta) = p\cdot f(x,\theta)$. 
\end{proof}

\begin{remark}
  In fact, one can show, under slightly stronger conditions (bounded first partial derivatives) that $\E_{\mu_0} H(x^*_t,p^*_t,\theta^*_t)$ is constant in time, using standard techniques (see e.g.,\, Sec. 4.2.9 of~\cite{liberzon2012calculus}).
\end{remark}

Let us now discuss the mean-field PMP. First, notice that it is a necessary condition, and hence is much weaker than the HJB characterization. Also, the PMP refers only to the open-loop control process $\bftheta$ with no explicit reference to an optimal control policy. Now, since the PMP is a necessary condition, we should discuss its relationship with classical necessary conditions in optimization. Equation~\eqref{eq:pmp_state} is simply the feed-forward ODE~\eqref{eq:cts_forward} under the optimal parameters $\bftheta^*$. On the other hand, Eq.~\eqref{eq:pmp_costate} defines the evolution of the co-state $p_s^*$. To draw analogy with constrained optimization, the co-state can be regarded as Lagrange multipliers which enforce the ODE constraint~\eqref{eq:cts_forward}. However, as in the proof of Thm.~\ref{thm:mean_field_PMP}, it may be more general to interpret it as the evolution of an adjoint variational condition backwards in time. 
The Hamiltonian maximization condition~\eqref{eq:pmp_max} is a unique feature of PMP-type statements, in that it does not characterize optimality in terms of vanishing of first order partial derivatives, as is the case in usual first order optimality conditions. Instead, optimal solutions must globally maximize the Hamiltonian function. This feature allows greater applicability since we can also deal with the case where the dynamics are not differentiable with respect to the controls/training weights, or when the optimal controls/training weights lie on the boundary of the set $\Theta$. Moreover, the usual first order optimality conditions and the celebrated back-propagation algorithm can be readily derived from the PMP, see~\cite{li2017maximum}. We note that compared to classical statements of the PMP~\cite{pontryagin1987mathematical}, the main difference in our result is the presence of the expectation over $\mu_0$ in the Hamiltonian maximization condition~\eqref{eq:pmp_max}. This is to be expected since the mean-field optimal control must depend on the distribution of input-target  pairs. 

We conclude the discussion by noting that the PMP above can be written more compactly as follows. For each control process $\bftheta \in L^\infty([0,T],\Theta)$, denote by $\bfx^\bftheta\coloneqq\{ x^\bftheta_t : 0\leq t\leq T\}$ and $\bfp^\bftheta\coloneqq\{ p^\bftheta_t : 0\leq t\leq T\}$ the solution of the Hamilton's equations~\eqref{eq:pmp_state} and~\eqref{eq:pmp_costate} using this control with the random variables $(x_0,y_0)\sim \mu_0$, i.e.
\begin{align}
  \begin{split}
      &\dot{x}^\bftheta_t = f(x^\bftheta_t, \theta_t),
      \qquad x^\bftheta_0 = x_0, \\ 
      &\dot{p}^\bftheta_t = -\nabla_x H(x^\bftheta_t, p^\bftheta_t, \theta_t),
      \qquad p^\bftheta_T = -\nabla_x \Phi(x^\bftheta_T, y_0).
  \end{split}
  \label{eq:compact_PMP_ivps}
\end{align}
Then, $\bftheta^*$ satisfies the PMP if and only if
\begin{align}
  \E_{\mu_0} H(x^{\bftheta^*}_t, p^{\bftheta^*}_t, \theta^*_t)
  \geq 
  \E_{\mu_0} H(x^{\bftheta^*}_t, p^{\bftheta^*}_t, \theta), 
  \quad \forall\,\theta \in \Theta.
  \label{eq:compact_PMP_max}
\end{align}
Furthermore, observe that the mean-field PMP derived above includes, as a special case, the necessary conditions for optimality for the sampled optimal control problem~\eqref{eq:sampled_ctrl_prob}. To see this, simply define the empirical measure $\mu^N_0 \coloneqq \frac{1}{N} \sum_{i=1}^{N} \delta_{(x^i_0,y^i_0)}$ and apply the mean-field PMP (Thm.~\ref{thm:mean_field_PMP}) with $\mu^N_0$ in place of $\mu_0$ to give
\begin{align}
  \frac{1}{N} \sum_{i=1}^{N}
  H(x^{\bftheta^*,i}_t, p^{\bftheta^*,i}_t, \theta^*_t)
  \geq 
  \frac{1}{N}\sum_{i=1}^N 
  H(x^{\bftheta^*,i}_t, p^{\bftheta^*,i}_t, \theta), 
  \quad \forall\,\theta \in \Theta,
  \label{eq:compact_PMP_max_sampled}
\end{align}
where each $\bfx^{\bftheta,i}$ and $\bfp^{\bftheta,i}$ are defined as in~\eqref{eq:compact_PMP_ivps}, but with the input-target pair $(x^i_0,y^i_0)$. Of course, since $\mu^N_0$ is a random measure, this is a random equation whose solution are random variables. 

\subsection{Connection between the HJB equation and the PMP}
\label{sec:principle_connection}

We now discuss some concrete connections between the HJB equation and the PMP, thus justifying our claim that the PMP can be understood as a local result compared to the global characterization of the HJB equation. 

It should be noted that the Hamiltonian defined in Pontryagin's maximum principle~\eqref{eq:pmp_H} is different from~\eqref{eq:lifted_H} in the HJB equation, due to different sign conventions in these two approaches of classical optimal control. We choose to keep this difference such that readers familiar with classical control theory can draw an analogy easily. Nevertheless, if one replaces $p,L,f$ in~\eqref{eq:pmp_H} by $-P,-\bar{L}, \bar{f}$ respectively and takes the infimum over $\Theta$ instead of the maximum condition in~\eqref{eq:pmp_max}, one formally obtains the negative of~\eqref{eq:lifted_H}. 

Now, our goal is to show that the HJB and PMP are more intimately connected than it appears in the definition of Hamiltonian. The deeper connections originate from the link between Hamilton's canonical equations (ODEs) and Hamilton-Jacobi equations (PDEs), of which we give an informal description as follows.

First, note that although the Hamiltonian dynamics~\eqref{eq:pmp_state} and~\eqref{eq:pmp_costate} describe the trajectory of particular random variables (completely determined by $(x_0,y_0)$), the optimality conditions are not dependent on the particular representation of the probability measures by these random variables. In other words, we could also formulate a maximum principle whose Hamiltonian flow is that on measures in a Wasserstein space, from which the above PMP can be seen as a ``lifting''. This approach would parallel the developments in the previous sections on the HJB equations. However, here we choose to establish and analyze the PMP in the lifted space due to the simplicity of having well-defined evolution equations. The corresponding evolution of measures would require more technical analysis while not being particularly more elucidating. Instead, we shall establish the connections by also lifting the HJB equation into $L^2(\Omega,\R^{d+l})$.

Consider the lifted HJB equation~\eqref{eq:HJB_lifted} in $L^2(\Omega,\R^{d+l})$. The key observation is that we can apply the method of characteristics (see e.g., Ch.~3.2 of~\cite{evans1998partial}) by defining $P_t=DV(t,\xi_t)$ and write down the characteristic evolution equations:
\begin{equation}
\begin{cases}
    \dot{\xi}_t = D_{P}\cH(\xi_t,P_t), \\
    \dot{P}_t = -D_{\xi}\cH(\xi_t,P_t).
\end{cases}
\label{eq:L2_LiftHamiltonSys}
\end{equation}

Suppose this system has a solution satisfying boundary conditions $\P_{\xi_0}=\mu_0, P_T=\nabla_w\bar{\Phi}(\xi_T)$, where the second condition comes from the terminal condition of~\eqref{eq:HJB_lifted}. To avoid technicalities, we further assume that the infimum in~\eqref{eq:lifted_H} is attained at $\theta^\dagger(\xi,P)$, which is always an interior point of $\Theta$. Hence~\eqref{eq:lifted_H} can be explicitly written down as

\begin{equation*}
    \cH = \E[P\cdot \bar{f}(\xi,\theta^\dagger(\xi,P))+\bar{L}(\xi,\theta^\dagger(\xi,P))],
\end{equation*}
and by first order condition we have
\begin{equation*}
    \E\left[\nabla_{\theta}\bar{f}(\xi,\theta^\dagger(\xi,P))P + \nabla_{\theta}\bar{L}(\xi,\theta^\dagger(\xi,P))\right]=0.
\end{equation*}
Plugging the above two equalities into~\eqref{eq:L2_LiftHamiltonSys} gives us
\begin{equation*}
\begin{cases}
    \dot{\xi}_t = \bar{f}(\xi_t,\theta^\dagger(\xi_t,P_t)), \\
    \dot{P}_t = -\nabla_w\bar{f}(\xi_t,\theta^\dagger(\xi_t,P_t))P_t -  \nabla_w\bar{L}(\xi_t,\theta^\dagger(\xi_t,P_t)).
\end{cases}
\end{equation*}

Let $\theta^*_t=\theta^\dagger(\xi_t,P_t)$. Note that $w=(x,y)$ is the concatenated variable and the last $l$ components of $\bar{f}$ are zero. If we only consider the first $d$ components, then we can deduce the $d-$dimensional dynamical system in $L^2(\Omega,\R^d)$:

\begin{equation}
\begin{cases}
    \dot{x}_t = f(\xi_t,\theta_t^*), \\
    \dot{p}_t = -\nabla_xf(x_t,\theta^*_t)p_t -  \nabla_x L(x_t,\theta^*_t).
\end{cases}
\label{eq:L2_HamiltonSys}
\end{equation}
If we make the transformation $p\rightarrow -p$ in Thm.~\ref{thm:mean_field_PMP}, it is straightforward to see that the deduced dynamical system by Thm.~\ref{thm:mean_field_PMP} satisfies~\eqref{eq:L2_HamiltonSys} in $L^2(\Omega,\R^d)$ and the boundary conditions are matched.

In summary, the Hamilton's equations~\eqref{eq:L2_HamiltonSys} in the PMP can be viewed as the characteristic equations for the HJB equation~\eqref{eq:HJB_lifted}. Consequently, the PMP pinpoints the necessary condition a characteristic of the HJB equation originating from (a random variable with law) $\mu_0$ must satisfy. This justifies the preceding claim that the PMP constitutes a local optimality condition as compared to the HJB equation.

\section{Small-time uniqueness}
\label{sec:small_T_uniqueness}

As discussed, the PMP constitute necessary conditions for optimality. A natural question is when are the PMP solutions also sufficient for optimality (See~\cite{bressan2007introduction}, Ch. 8 for some discussions on sufficiency). One simple case where it is sufficient, assuming an optimal solution exists, is when the PMP equations admit a unique solution. In this section, we investigate the uniqueness properties of the PMP system. 

Note that even if there exists a unique solution $\theta^\dagger(\nu)$ of the Hamiltonian maximization $\argmax_\theta \E_{(x,p)\sim\nu}H(x,p,\theta)$ for any $\P_{(x,p)}$, the equation~\eqref{eq:compact_PMP_ivps} reduces to a highly non-linear two-point boundary value problem for $\bfx^*, \bfp^*$, further coupled with their laws. Even without the coupling to laws, such two-point boundary value problems are known to not have unique solutions in general (see e.g., Ch. 7 of~\cite{kelley2010theory}). In the following, we shall show that if $T$ is sufficiently small and $H$ is strongly concave, then the PMP admits a unique solution. Hereafter, we retain assumption (A2$'$) and replace (A1$'$) with a stronger assumption, which greatly simplifies our arguments:
\begin{itemize}
    \item[(A1$''$)] $f$ is bounded; $f,L,\Phi$ are twice continuously differentiable with respect to both $x,\theta$, with bounded and Lipschitz partial derivatives. 
\end{itemize}

With an estimate of the difference in flow maps due to two different controls, we can prove a small-time uniqueness result for the PMP. 
\begin{theorem}
\label{thm:small_T_uniqueness}
  Suppose that $H(x,p,\theta)$ is strongly concave in $\theta$, uniformly in $x,p\in\R^d$, i.e.\,$\nabla^2_{xx}H(x,p,\theta) + \lambda_0 I \preceq 0$ for some $\lambda_0>0$. Then, for sufficiently small $T$, if $\bftheta^1$ and $\bftheta^2$ are solutions of the PMP~\eqref{eq:compact_PMP_max} 
  then $\bftheta^1  = \bftheta^2$.
\end{theorem}

Note that since we are considering the effects of $T$, in the rest of the estimates in this section, the dependence of constants on $T$ are explicitly considered. We first estimate the difference of flow-maps driven by two different controls.

\begin{lemma}
\label{lem:small_T_xp_lipschitz}
  Let $\bftheta^1, \bftheta^2\in L^\infty([0,T],\Theta)$. Then, there exists a constant $T_0$ such that for all $T\in[0,T_0)$, we have 
  \begin{align*}
      \Vert \bfx^{\bftheta^1} - \bfx^{\bftheta^2} \Vert_{L^\infty} 
      + \Vert \bfp^{\bftheta^1} - \bfp^{\bftheta^2} \Vert_{L^\infty}
      \leq C(T) \Vert \bftheta^1 - \bftheta^2 \Vert_{L^\infty}. 
  \end{align*}
  where $C(T)>0$ satisfies $C(T)\rightarrow 0$ as $T\rightarrow 0$.
\end{lemma}
\begin{proof}
    Denote $\bfdelta\bftheta\coloneqq\bftheta^1 - \bftheta^2$, $\bfdelta\bfx \coloneqq \bfx^{\bftheta^1} - \bfx^{\bftheta^2}$ and $\bfdelta\bfp \coloneqq \bfp^{\bftheta^1} - \bfp^{\bftheta^2}$.  
    Since $x^{\bftheta^1}_0=x^{\bftheta^2}_0=x_0$, integrating the respective ODEs and using (A1$''$) we have
    \begin{align*}
        \Vert \delta x_t \Vert 
        \leq \int_{0}^{t} \Vert f(x^{\bftheta^1}_s, \theta^1_s) - f(x^{\bftheta^1}_t, \theta^2_s) \Vert ds 
        \leq K_L \int_{0}^{T} \Vert \delta x_s \Vert ds 
        + K_L \int_{0}^{T} \Vert \delta p_s \Vert ds, \\       
    \end{align*}
    and so 
    \begin{align*}
        \Vert \bfdelta\bfx \Vert_{L^\infty}
        \leq K_L T \Vert \bfdelta\bfx \Vert_\infty + K_L T \Vert \bfdelta \theta\Vert_\infty.
    \end{align*}
    Now, if $T<T_0\coloneqq1/K_L$, we then have
    \begin{align}
        \Vert \bfdelta \bfx \Vert_{L^\infty}
        \leq \frac{K_L T}{1 - K_L T}
        \Vert \bfdelta \bftheta \Vert_{L^\infty}.
        \label{eq:delta_x_est}
    \end{align}
    Similarly, 
    \begin{align*}
        &\Vert \delta p_t \Vert \leq K_L \Vert \delta x_T \Vert
        + K_L \int_{t}^{T} \Vert \delta x_s \Vert ds
        + K_L \int_{t}^{T} \Vert \delta p_s \Vert ds, \\ 
        &\Vert \bfdelta \bfp \Vert_{L^\infty}
        \leq (K_L + K_L T) \Vert \bfdelta \bfx \Vert_{L^\infty}
        + K_L T \Vert \bfdelta \bfp \Vert_{L^\infty},
    \end{align*}
    and hence
    \begin{align}
        \Vert \bfdelta \bfp \Vert_{L^\infty}
        \leq \frac{K_L(1+T)}{1-K_L T} \Vert \bfdelta \bfx \Vert_{L^\infty}.
        \label{eq:delta_p_est}
    \end{align}
    Combining~\eqref{eq:delta_x_est} and~\eqref{eq:delta_p_est} proves the claim. 
\end{proof}

With the above estimate, we can now prove Thm.~\ref{thm:small_T_uniqueness}. 
\begin{proof}[Proof of Thm.~\ref{thm:small_T_uniqueness}]
    By uniform strong concavity, the function $\theta \mapsto \E_{\mu_0} H(x^{\bftheta^1}_t,x^{\bftheta^1}_t,\theta)$ is strongly concave. Thus, we have a $\lambda_0>0$ such that
    \begin{align*}
        \frac{\lambda_0}{2} \Vert \theta^1_t - \theta^2_t \Vert^2
        \leq 
            {\left[
                \E_{\mu_0} \nabla H(x^{\bftheta^1}_t, p^{\bftheta^1}_t, \theta^2_t) 
                - \E_{\mu_0} \nabla H(x^{\bftheta^1}_t, p^{\bftheta^1}_t, \theta^1_t) 
            \right]} \cdot (\theta^1_t - \theta^2_t).
    \end{align*}
    A similar expression holds for $\theta \mapsto \E_{\mu_0} H(x^{\bftheta^2}_t,x^{\bftheta^2}_t,\theta)$ and so combining them and using assumptions (A1$''$) we have
    \begin{align*}
        \lambda_0 \Vert \theta^1_t - \theta^2_t \Vert^2
        \leq& 
            {\left[
                \E_{\mu_0} \nabla H(x^{\bftheta^1}_t, p^{\bftheta^1}_t, \theta^2_t) 
                - \E_{\mu_0} \nabla H(x^{\bftheta^1}_t, p^{\bftheta^1}_t, \theta^1_t) 
            \right]} \cdot (\theta^1_t - \theta^2_t) \\ 
        &+
        {\left[
            \E_{\mu_0} \nabla H(x^{\bftheta^2}_t, p^{\bftheta^2}_t, \theta^1_t) 
            - \E_{\mu_0} \nabla H(x^{\bftheta^2}_t, p^{\bftheta^2}_t, \theta^2_t) 
        \right]} \cdot (\theta^1_t - \theta^2_t) \\
        \leq&
        \E_{\mu_0} 
        \Vert 
            \nabla H(x^{\bftheta^1}_t, p^{\bftheta^1}_t, \theta^1_t) 
            - \nabla H(x^{\bftheta^2}_t, p^{\bftheta^2}_t, \theta^1_t) 
        \Vert
        \Vert 
            \theta^1_t - \theta^2_t
        \Vert \\ 
        &+\E_{\mu_0} 
        \Vert 
            \nabla H(x^{\bftheta^1}_t, p^{\bftheta^1}_t, \theta^2_t) 
            - \nabla H(x^{\bftheta^2}_t, p^{\bftheta^2}_t, \theta^2_t) 
        \Vert
        \Vert 
            \theta^1_t - \theta^2_t
        \Vert \\
        \leq&
        K_L \Vert \bfdelta\bftheta \Vert_{L^\infty} (\Vert \bfdelta\bfx \Vert_{L^\infty} + \Vert \bfdelta\bfp \Vert_{L^\infty}).
    \end{align*}
    Combining the above and Lemma~\ref{lem:small_T_xp_lipschitz}, we have
    \begin{align*}
        \Vert \bfdelta\bftheta \Vert_{L^\infty}^2
        \leq \frac{K_L}{\lambda_0} C(T)
        \Vert \bfdelta\bftheta \Vert_{L^\infty}^2.
    \end{align*}
    But $C(T)=o(1)$ and so we may take $T$ sufficiently small so that $K_L C(T) < \lambda_0$ to conclude that $\Vert \bfdelta\bftheta \Vert_{L^\infty} = 0$. 
\end{proof}

In the context of machine learning, since $f$ is bounded, small $T$ roughly corresponds to the regime where the reachable set of the forward dynamics is small. This can be loosely interpreted as the case where the model has low capacity or expressive power. We note that the number of parameters is still infinite, since we only require $\bftheta$ to be essentially bounded and measurable in time. Hence, Thm.~\ref{thm:small_T_uniqueness} can be interpreted as the statement that when the model capacity is low, the optimal solution is unique, albeit with possibly high loss function values. Note that the strong concavity of the Hamiltonian does not imply that the loss function $J$ is strongly convex, or even convex, which is often an unrealistic assumption in deep learning. In fact, in the case considered in Example \ref{ex:linear_with_theta}, we observe that $H$ is strongly concave but the loss function $J$ can be highly non-convex due to the non-linear transformation $\sigma$. Compared with the characterization using HJB (Sec.~\ref{sec:HJB_vicosity}), we observe that the uniqueness of the solutions of the PMP requires the small $T$ condition. 


\section{From mean-field PMP to sampled PMP}
\label{sec:error_analysis}

So far, we have focused our discussion on the mean-field control problem~\eqref{eq:mean_field_ctrl_prob} and mean-field PMP~\eqref{eq:compact_PMP_max}. However, the solution of the mean-field PMP requires maximizing an expectation. Hence, in practice we must resort to solving a sampled version~\eqref{eq:compact_PMP_max_sampled}, which constitutes necessary conditions for the sampled optimal control problem~\eqref{eq:sampled_ctrl_prob}.

The goal of this section is to draw some precise connections between the solutions of the mean-field PMP~\eqref{eq:compact_PMP_max} and the sampled PMP~\eqref{eq:compact_PMP_max_sampled}. In particular, we show that under appropriate conditions, near any stable (to be precisely defined later) solution of the mean-field PMP~\eqref{eq:compact_PMP_max} we can find with high probability a solution of the sampled PMP~\eqref{eq:compact_PMP_max_sampled}. This allows us to establish a concrete link, via the maximum principle, between solutions of the population risk minimization problem~\eqref{eq:mean_field_ctrl_prob} and the empirical risk minimization problem~\eqref{eq:sampled_ctrl_prob}. 
To proceed, the key observation is that the interior solutions to both the mean-field and sampled PMPs can be written as the solutions to algebraic equations on Banach spaces. Indeed, in view of the compact notation~\eqref{eq:compact_PMP_max}, let us suppose that $\bftheta^*$ is a solution of the PMP such that the maximization step attains a maximum in the interior of $\Theta$ for a.e.\,$t\in[0,T]$. Note that if $\Theta$ is sufficiently large, e.g., $\Theta=\R^m$, then this must be the case. We shall hereafter assume this holds. Consequently, the PMP solution satisfies (by dominated convergence theorem)
\begin{align}
  {\bfF(\bftheta^*)}_t \coloneqq \E_{\mu_0} \nabla_\theta H(x^{\bftheta^*}_t, p^{\bftheta^*}_t , \theta^*_t) = 0,
  \label{eq:F_defn}
\end{align}
for a.e.\,$t$, where $\bfF: L^\infty([0,T],\Theta) \rightarrow L^\infty([0,T],\R^m)$ is a Banach space mapping. 
Similarly, from~\eqref{eq:compact_PMP_max_sampled} we know that an interior solution $\bftheta^N$ of the finite-sample PMP is a random variable which satisfies
\begin{align}
    {\bfF_N(\bftheta^N)}_t \coloneqq \frac{1}{N} \sum_{i=1}^{N}
    \nabla_\theta H(x^{\bftheta^N,i}_t, p^{\bftheta^N,i}_t , \theta^N_t) = 0,
\end{align}
for a.e.\,$t$. Now, $\bfF_N$ is a random approximation of $\bfF$ and $\E \bfF_N(\bftheta) = \bfF(\bftheta)$ for all $\bftheta$. In fact, $\bfF_N \rightarrow \bfF$ almost surely by law of large numbers. Hence, the analysis of the approximation properties of the mean-field PMP by its sampled counterpart amounts to the study of the approximation of zeros of $\bfF$ by those of $\bfF_N$. 

In view of this, we shall take a brief excursion to develop some theory on random approximations of zeros of Banach space mappings at an abstract level, and then use these results to deduce properties of the PMP approximations. The techniques employed in the next section are reminiscent of classical numerical analysis results on finite difference approximation schemes~\cite{keller1975approximation}, except that we work with random approximations. 

\subsection{Excursion: random approximations of zeros of Banach space mappings}
\label{sec:banach_space_excusion}

Let $(U,\Vert\cdot\Vert_U),(V,\Vert\cdot\Vert_V)$ be Banach spaces and $F: U \rightarrow V$ be a mapping. We first define a notion of stability, which shall be a primary condition that ensures existence of close-by zeros of approximations.  
\begin{definition}
\label{def:stability}
	For $\rho>0$ and $x\in U$, define $S_\rho(x)\coloneqq\{y\in U: \Vert x-y \Vert_U \leq \rho \}$. We say that the mapping $F$ is stable on $S_\rho(x)$ if there exists a constant $K_\rho>0$ such that for all $y,z \in S_\rho(x)$,
	\begin{align*}
		\Vert y - z \Vert_U \leq K_\rho \Vert F(y) - F(z) \Vert_V .
	\end{align*}
\end{definition}
Note that if $F$ is stable on $S_\rho(x)$, then it is trivially true that it has at most one solution to $F=0$ on $S_\rho(x)$. If it does have a solution, say at $x^*$, then it is necessarily isolated, i.e., if $DF(x^*)$ exists, then it is non-singular. 
The following proposition establishes a stronger version of this: if $DF(x)$ exists for any $x\in S_{\rho}(x^*)$, then it is necessarily non-singular.

\begin{proposition}
\label{prop:stability_implies_isolation}
  Let $F$ on $S_\rho(x^*)$ be stable. Then, for any $x\in S_\rho(x^*)$, if $DF(x)$ exists, then it is non-singular, i.e. $DF(x)y = 0$ implies $y=0$.
\end{proposition}
\begin{proof}
  Suppose for the sake of contradiction that $DF(x)y = 0$ and $\Vert y\Vert_U \neq 0$. Define $z(\alpha) \coloneqq x + \alpha y$ with $\alpha$ sufficiently small so that $z(\alpha)\in S_\rho(x^*)$. Then, 
  \begin{align*}
    \alpha \Vert y \Vert_U = & \Vert x - z(\alpha) \Vert_U              \\
    \leq                   & K_\rho \Vert F(x) - F(z(\alpha)) \Vert_V \\ 
    \leq                   & K_\rho ( \alpha \Vert DF(x)y \Vert_V     
    + \Vert F(x+\alpha y) - F(x) - DF(x) \alpha y \Vert_V ).
  \end{align*}
  But $DF(x)y=0$, and so $\alpha \Vert y \Vert_U \leq K_\rho r(x,\alpha y) \alpha \Vert y \Vert_U$, By definition of the Fr\'{e}chet derivative~\eqref{eq:frechet_deriv_defn}, $r(x,\alpha y)\rightarrow 0$ as $\alpha\rightarrow 0$. Thus if $\alpha$ is sufficiently small so that $K_\rho r(x,\alpha y)<1$, then $\Vert y \Vert_U=0$ and hence we arrive at a contradiction. 
\end{proof}


As the previous proposition suggests, a converse statement that establishes stability will require $DF(x)$ to be non-singular on some neighborhood of $x^*$. One in fact requires more, i.e. that $DF$ needs to be Lipschitz. Note that for a linear operator $A:U\rightarrow V$, we also use $\Vert A \Vert_V$ to denote the usual induced norm, $\Vert A \Vert_V = \sup_{\Vert y\Vert_U \leq 1} \Vert Ay \Vert_V$.

\begin{proposition}
\label{prop:isolation_plus_lip_implies_stability}
    Suppose $DF(x^*)$ is non-singular, $DF(x)$ exists and $\Vert DF(x) - DF(y) \Vert_V \leq K_L \Vert x - y \Vert_U$ for all $x,y\in S_{\rho}(x^*)$. Then, $F$ is stable on $S_{\rho_0}(x^*)$ for any $0<\rho_0\leq \min(\rho, \tfrac{1}{2}(K_L \Vert {DF(x^*)}^{-1} \Vert_U)^{-1})$ with stability constant
    \begin{align*}
        K_{\rho_0} = 2\Vert {DF(x^*)}^{-1} \Vert_U.
    \end{align*}
\end{proposition}
\begin{proof}
  Let $\rho_0\leq\rho$ and take $x,y \in S_{\rho_0}(x^*)$. Using the mean value theorem, we can write $F(x)-F(y) = R(x,y)(x-y)$ where
  \begin{align*}
      R(x,y) \coloneqq \int_{0}^{1} DF(s x + (1-s) y) ds.
  \end{align*}
  But, using the Lipschitz condition we have
  \begin{align*}
      \Vert R(x,y) - DF(x^*) \Vert_V
      \leq&
      \int_{0}^{1} \Vert DF(s x + (1-s) y) - DF(s x^* + (1-s) x^*) \Vert_V ds \\ 
      \leq& K_L
      \int_{0}^{1} \Vert s (x-x^*) + (1-s) (y-x^*) \Vert_U ds \\ 
      \leq& \rho_0 K_L.
  \end{align*}
  We take $\rho_0$ sufficiently small so that $\rho_0 K_L \leq \tfrac{1}{2} \Vert {DF(x^*)}^{-1} \Vert_U^{-1}$. Then, by the Banach lemma, $R(x,y)$ is non-singular and $\Vert {R(x,y)}^{-1} \Vert_U \leq 2 \Vert {DF(x^*)}^{-1}\Vert_U$. The result follows since $(x-y) = {R(x,y)}^{-1} (F(x) - F(y))$. 
\end{proof}

Now, let us now introduce a family of \emph{random} mappings $F_N$ that approximate $F$. Let $(\Omega, \mathcal{F}, \mathbb{P})$ be a probability space and $\{ F_N(\omega): N \geq 1, \omega\in\Omega \}$ be a family of mappings from $U$ to $V$ such that $\omega \mapsto F_N(\omega)(x)$ is $\mathcal{F}$-measurable for each $x$ (we equip the Banach spaces $U,V$ with the Borel $\sigma$-algebra). We make the following assumptions which will allow us to relate the random solutions of $F_N=0$ with those of $F=0$ in Thm.~\ref{thm:banach_finite_sample}.
\begin{itemize}
  \item[(B1)] (Stability) There exists $x^*\in U$ such that $F(x^*)=0$ and $F$ is stable on $S_\rho(x^*)$ for some $\rho>0$. 
	\item[(B2)] (Uniform convergence in probability) For all $N\geq 1$, $DF(x)$ and $DF_N(x)$ exists for all $x\in S_\rho(x^*)$, $\PP$-a.s.\,and
    \begin{align*}
    	  & \PP 
    	\left[
    	\Vert F(x) - F_N(x)\Vert_V \geq s
    	\right]
    	\leq r_1(N,s), \\ 
    	  & \PP 
    	\left[
    	\Vert DF(x) - DF_N(x)\Vert_V \geq s
    	\right]
    	\leq r_2(N,s),
    \end{align*}
	   for some real-valued functions $r_1,r_2$ such that $r_1(N,s),r_2(N,s)\rightarrow 0$ as $N\rightarrow\infty$.
	\item[(B3)] (Uniformly Lipschitz derivative) There exists $K_L>0$ such that for all $x,y \in S_\rho(x^*)$, 
    \begin{align*}
    	\Vert DF_N(x) - DF_N(y) \Vert_V \leq K_L \Vert x - y \Vert_U, 
    	\qquad \PP\text{-a.s.}                                   
    \end{align*}
\end{itemize}

\begin{theorem}
\label{thm:banach_finite_sample}
	Let (B1)-(B3) hold. Then, there exist positive constants $s_0, \rho_1, C$ with $\rho_1<\rho$ and $U$-valued random variables $x_N\in S_{\rho_1}(x^*)$ satisfying
	\begin{align*}
		  & \PP [\Vert x_N - x^* \Vert_U \geq C s ] \leq r_1(N, s) + r_2(N, s), \qquad s \in (0,s_0],\\ 
		  & \PP [F_N(x_N) \neq 0] \leq r_1(N, s_0) + r_2(N, s_0).
    \end{align*}
    In particular, $x_N \rightarrow x^*$ and $F_N(x_N) \rightarrow 0$ in probability. 
\end{theorem}

To establish Thm.~\ref{thm:banach_finite_sample}, we first prove that for large $N$, with high probability $DF_N(x^*)$ is non-singular and $\Vert DF_N(x^*)^{-1} \Vert_U$ is uniformly bounded. 

\begin{lemma}
\label{lem:non_singular_DF_N}
  Let (B1)-(B3) hold. Then, there exists a constant $s_0>0$ such that for each $s\in (0,s_0]$ and $N\geq 1$, there exists a measurable $A_N(s)\subset \Omega$ such that $\PP[A_N(s)] \geq 1 - r_1(N,s) - r_2(N,s)$ and for each $\omega\in A_N(s)$,
  \begin{align*}
      \Vert F(x^*) - F_N(\omega)(x^*) \Vert_V < s.
  \end{align*}
  Moreover, $DF_N(\omega)(x^*)$ is non-singular with
\begin{align*}
      \Vert {DF_N(\omega)(x^*)}^{-1} \Vert_U \leq 2 \Vert {DF(x^*)}^{-1}\Vert_U. 
  \end{align*}
  In particular, $DF_N(\omega)$ is stable on $S_{\rho_0}(x^*)$ with $\rho_0\leq \min(\rho, \tfrac{1}{4}{(K_L \Vert {DF(x^*)}^{-1} \Vert_U)}^{-1})$ and stability constant $K_{\rho_0} = 4\Vert {DF(x^*)}^{-1}\Vert_U$.
\end{lemma}

\begin{proof}
  For $s>0$ set
  \begin{align*}
      A_N(s) \coloneqq 
      \{
          &\omega\in\Omega:
          \Vert F(x^*) - F_N(\omega)(x^*) \Vert_V) < s \\
          &\text{and }
          \Vert DF(x^*) - DF_N(\omega)(x^*) \Vert_V < s            
      \}.
  \end{align*}
  Observe that $A_N(s)$ is measurable as $DF_N(\omega)(x^*)$ is measurable and assumption (B2) implies $\PP[A_N(s)]\geq 1 - r_1(N,s) - r_2(N,s)$. Now, take $s$ sufficiently small so that $s \leq s_0 = \tfrac{1}{2} \Vert {DF(x^*)}^{-1}\Vert_U^{-1}$. Then, for each $\omega\in A_N(s)$, the Banach lemma implies $DF_N(\omega)(x^*)$ is non-singular and
\begin{align*}
  \Vert DF_N(\omega)(x^*)^{-1} \Vert_U                     
  \leq \frac{\Vert DF(x^*)^{-1} \Vert_U}{ 1 - \frac{1}{2}} 
  = 2 \Vert DF(x^*)^{-1} \Vert_U.
  \end{align*}
  Finally, we use Proposition~\ref{prop:isolation_plus_lip_implies_stability} to deduce stability of $F_N(\omega)$.
\end{proof}
Now we are ready to prove Thm.~\ref{thm:banach_finite_sample} by constructing a uniform contraction mapping whose fixed point is a solution of $F_N(x)=0$. 
\begin{proof}[Proof of Thm.~\ref{thm:banach_finite_sample}]
  Let $s_0$, $A_N(s)$ and $\rho_0$ be those defined in Lemma~\ref{lem:non_singular_DF_N}. 
  For each $\omega \in A_N(s)$ with $s\leq s_0$, define the mapping
  \begin{align*}
      G_N(\omega)(x) \coloneqq x - {DF_N(\omega)(x^*)}^{-1} F_N(\omega)(x). 
  \end{align*}
  We now show that this is in fact a uniform contraction on $S_{\rho_1}(x^*)$ for sufficiently small $\rho_1$. Let $x,y\in S_{\rho_1}(x^*)$. By the mean value theorem, we have
  \begin{align*}
      & G_N(\omega)(x) - G_N(\omega)(y) \\
      =\, & {DF_N(\omega)(x^*)}^{-1}
      [ DF_N(\omega)(x^*)(x-y) - (F_N(\omega)(x) - F_N(\omega)(y)) ] \\ 
      =\, & {DF_N(\omega)(x^*)}^{-1} [DF_N(\omega)(x^*) - R_N(\omega)(x,y)] (x - y),
  \end{align*}
  where $R_N(\omega)(x,y) = \int_{0}^{1} DF_N(\omega)( s x + (1-s) y ) ds$.
      Lipschitz condition (B3) implies
  \begin{align*}
      \Vert DF_N(\omega)(x^*) - R_N(\omega)(x,y) \Vert_V  \leq \rho_1 K_L
  \end{align*}
  and hence by Lemma~\ref{lem:non_singular_DF_N},
  \begin{align*}
      \Vert G_N(\omega)(x) - G_N(\omega)(y) \Vert_U \leq \alpha \Vert x - y \Vert_U, 
  \end{align*}
  where $\alpha = 2 K_L \rho_1 \Vert DF(x^*)^{-1} \Vert_U$. We now pick $\rho_1<\rho_0$ sufficiently small so that $\alpha < 1$. It remains to show that the mapping $G_N(\omega)$ maps $S_{\rho_1}(x^*)$ onto itself. Let $x \in S_{\rho_1}(x^*)$, then by noting that $F(x^*)=0$,
  \begin{align*}
      \Vert G_N(\omega)(x) - x^* \Vert_U 
          & \leq \Vert G_N(\omega)(x) - G_N(\omega)(x^*) \Vert_U
      + \Vert G_N(\omega)(x^*) - x^* \Vert_U \\ 
          & \leq \alpha \rho_1 + 2\Vert {DF(x^*)}^{-1} \Vert_U \Vert F_N(\omega)(x^*) - F(x^*) \Vert_V .
  \end{align*}
  Using Lemma~\ref{lem:non_singular_DF_N} again, we have 
  \begin{align*}
      \Vert G_N(\omega)(x) - x^* \Vert_U
          & \leq \alpha \rho_1 + 2 s \Vert {DF(x^*)}^{-1} \Vert_U.
  \end{align*}
  We now take $s_0>s$ small enough so that $2 s_0 \Vert {DF(x^*)}^{-1} \Vert_U < (1-\alpha)\rho_1$. Then, for all $N\geq 1$, $G_N(\omega)$ is a contraction, uniform in $N$, on $S_{\rho_1}(x^*)$ and hence by Banach fixed point theorem, there exists a unique $\tilde{x}_{N,s}(\omega) \in S_{\rho_1}(x^*)$ such that $G_N(\omega)(\tilde{x}_{N,s}(\omega))=\tilde{x}_{N,s}(\omega)$, i.e. $F_N(\omega)(\tilde{x}_{N,s}(\omega))=0$ for all $\omega \in A_N(s)$. Moreover, $\tilde{x}_{N,s}(\omega) = \lim_{k\rightarrow\infty} [G_N(\omega)]^{(k)}(y)$ for any $y\in S_{\rho_0}(x^*)$. 
  Define 
  \begin{align*}
      x_{N,s}(\omega) = \mathbf{1}_{A_N(s)}(\omega) \tilde{x}_N(\omega) 
      + \mathbf{1}_{A_N(s)^c}(\omega) x^*.                          
  \end{align*}
  Now, $x_{N,s}$ is measurable since $A_N(s)$ is measurable and $\tilde{x}_{N,s}$ is the limit of measurable random variables, and hence measurable. Moreover, $A_N(s)\subset \{ F_N(x_N) = 0 \}$ and so $\PP[F_N(x_{N,s})=0] \geq 1 - r_1(N,s) - r_2(N,s)$. Since $x_{N,s} \in S_{\rho_1}(x^*)$ and $\rho_1 < \rho_0$, using the stability of $F_N(\omega)$ established in Lemma~\ref{lem:non_singular_DF_N}, and the fact that $F_N(x_{N,s})=F(x^*)=0$, we have for any $\omega \in A_{N}(s)$
  \begin{align*}
      \Vert x_{N,s}(\omega) - x^* \Vert_U
      &\leq K_{\rho_0} \Vert F_N(\omega)(x_{N,s}) - F_N(\omega)(x^*) \Vert_V \\ 
      &\leq 4 \Vert {DF(x^*)}^{-1} \Vert_U \Vert F(x^*) - F_N(\omega)(x^*) \Vert_V \\ 
      &< 4 s \Vert {DF(x^*)}^{-1} \Vert_U,
  \end{align*}
  and so $\PP[\Vert x_{N,s}(\omega) - x^* \Vert_U \geq C s] \leq r_1(N,s)+r_2(N,s)$ with $C =  4 \Vert {DF(x^*)}^{-1} \Vert_U$. At this point, it appears that $x_{N,s}$ depends on $s$. However, notice that for all $s\leq s_0$, $A_N(s)\subset A_N(s_0)$. But, $x_{N,s}(\omega)$ is the unique solution of $F_N(\omega)(\cdot)=0$ in $S_{\rho_1}(x^*)$ for each $\omega\in A_N(s) \subset A_N(s_0)$. Therefore, $x_{N,s}(\omega) = x_{N,s_0}(\omega)$ for all $s\leq s_0$. We can thus write $x_N \coloneqq x_{N,s_0} \equiv x_{N,s}$. 

  Lastly, convergence in probability follows from the decay of the functions $r_1,r_2$ has $N\rightarrow\infty$. 
\end{proof}

\subsection{Error estimate for sampled PMP}

Now, our goal is to apply the theory developed in Sec.~\ref{sec:banach_space_excusion} to the PMP. We shall assume that $\bftheta^*$, the solution of the mean-field PMP, is such that $\bfF(\bftheta^*)=0$ (recall that this holds for $\Theta=\R^m$). Suppose further that $\bfF$ is stable at $\bftheta^*$ (see Def.~\ref{def:stability}). We wish to show that for sufficiently large $N$, with high probability $\bfF_N$ must have a solution $\bftheta^N$ close to $\bftheta^*$. 

In view of Thm.~\ref{thm:banach_finite_sample}, we only need to check that (B2)-(B3) are satisfied. This requires a few elementary estimates and an application of the infinite-dimensional Hoeffding's inequality~\cite{pinelis1986remarks}. 
\begin{lemma}
\label{lem:lipschitz_xp}
  There exist constants $K_B, K_L>0$ such that for all $\bftheta,\bfphi \in L^{\infty}([0,T],\Theta)$
  \begin{align*}
      \Vert \bfx^\bftheta \Vert_{L^\infty}
      + \Vert \bfp^\bftheta \Vert_{L^\infty}
      & \leq K_B,
      \\ 
      \Vert \bfx^\bftheta - \bfx^\bfphi \Vert_{L^\infty}
      + \Vert \bfp^\bftheta - \bfp^\bfphi \Vert_{L^\infty}
      &\leq K_L \Vert \bftheta - \bfphi \Vert_{L^\infty}. 
  \end{align*}
\end{lemma}

\begin{proof}
  We have by Gronwall's inequality for a.e. $t$,
  \begin{align*}
      \Vert x^{\bftheta}_t - x^{\bfphi}_t \Vert
      &= 
      \left\Vert
          \int_{0}^{t}  
          f(x^{\bftheta}_s, \theta_s) - f(x^{\bftheta}_s, \theta_s) ds
      \right\Vert \\ 
      &\leq K_L \int_{0}^{t} \Vert x^{\bftheta}_s - x^{\bfphi}_s \Vert ds
      + K_L \int_{0}^{t} \Vert \theta_s - \phi_s \Vert ds \\ 
      &\leq K_L T e^{K_L T} \Vert \bftheta - \bfphi \Vert_{L^\infty}.
  \end{align*}
  Similarly,
  \begin{align*}
      \Vert p^{\bftheta}_t - p^{\bfphi}_t \Vert
      \leq& \Vert \nabla_x \Phi(x^{\bftheta}_T,y_0) - \nabla_x \Phi(x^{\bfphi}_T,y_0) \Vert \\ 
      &+ \left\Vert
          \int_{t}^{T} \nabla_x H(x^{\bftheta}_s, p^\bftheta_s, \theta_s) 
          - \nabla_x H(x^{\bfphi}_s, p^{\bfphi}_s, \phi_s)
          ds
      \right\Vert \\ 
      \leq& K_L \Vert x^{\bftheta}_T - x^{\bfphi}_T\Vert + K_L
      \int_{t}^{T} \Vert x^{\bftheta}_s - x^{\bfphi}_s \Vert ds 
      + K_L \int_{t}^{T} \Vert p^{\bftheta}_s - p^{\bfphi}_s \Vert ds \\ 
      \leq& (K_L+T)K_L T e^{2 K_L T} \Vert \bftheta - \bfphi \Vert_{L^\infty}.
  \end{align*}
\end{proof}

Notice that we can view $\bfx^{\bftheta} \equiv \bfx(\bftheta)$ as a Banach space mapping from $L^\infty([0,T],\Theta)$ to $L^\infty([0,T],\R^d)$, and similarly for $\bfp^{\bftheta}$. Below, we establish some elementary estimates for the derivatives of these mappings with respect to $\bftheta$.
\begin{lemma}
\label{lem:lipschitz_Dxp}
  There exist constants $K_B, K_L>0$ such that for all $\bftheta,\bfphi \in L^{\infty}([0,T],\Theta)$
  \begin{align*}
      \Vert D \bfx^\bftheta \Vert_{L^\infty} + \Vert D \bfp^\bftheta \Vert_{L^\infty}
      &\leq K_B, \\ 
      \Vert D \bfx^\bftheta - D \bfx^\bfphi \Vert_{L^\infty}
      + \Vert D \bfp^\bftheta - D \bfp^\bfphi \Vert_{L^\infty}
      &\leq K_L \Vert \bftheta - \bfphi \Vert_{L^\infty}. 
  \end{align*}
\end{lemma}
\begin{proof}
  Let $\bfeta \in L^\infty([0,T],\R^m)$ such that $\Vert\bfeta\Vert_{L^\infty}\leq 1$. For brevity, let us also denote $f^\bftheta_t \coloneqq f(x^\bftheta_t, \theta_t)$ and $H^\bftheta_t \coloneqq H(x^\bftheta_t, p^\bftheta_t, \theta_t)$. Then, $(D\bfx^\bftheta)\bfeta$ satisfy the linearized ODE
  \begin{align*}
      \frac{d}{dt} {[(D\bfx^\bftheta)\bfeta]}_t = 
      \nabla_x f^\bftheta_t {[(D\bfx^\bftheta)\bfeta]}_t
      + \nabla_\theta f^\bftheta_t \eta_t,
      \qquad {[(D\bfx^\bftheta)\bfeta]}_0 = 0.
  \end{align*}
  Gronwall's inequality and (A1$''$) immediately implies that $\Vert {[(D\bfx^\bftheta)\bfeta]}_t \Vert \leq K_L \Vert \bfeta \Vert_{L^\infty}$, and so $\Vert D\bfx^\bftheta \Vert_{L^\infty}\leq K'$. Next,
  \begin{align*}
      \Vert {[(D\bfx^\bftheta)\bfeta]}_t - {[(D\bfx^\bfphi)\bfeta]}_t \Vert
      \leq& 
      \int_{0}^{t} \Vert \nabla_x f^\bftheta_s \Vert
      \Vert {[(D\bfx^\bftheta)\bfeta]}_s - {[(D\bfx^\bfphi)\bfeta]}_s \Vert ds \\ 
      &+ \int_{0}^{t} \Vert \nabla_x f^\bftheta_s - \nabla_x f^\bfphi_s \Vert
      \Vert {[(D\bfx^\bfphi)\bfeta]}_s \Vert ds \\
      &+ \int_{0}^{t} \Vert \nabla_\theta f^\bftheta_s 
      - \nabla_\theta f^\bfphi_s \Vert
      \Vert \eta_s \Vert ds.
  \end{align*}
  But, using Lemma~\ref{lem:lipschitz_xp}, assumption (A1$''$), we have
  \begin{align*}
      \Vert \nabla_x f^\bftheta_s - \nabla_x f^\bfphi_s \Vert
      \leq& K_L \Vert x^\bftheta_s - x^\bfphi_s \Vert + K_L \Vert \theta_s - \phi_s \Vert \\ 
      \leq& K_L \Vert \bftheta - \bfphi \Vert_{L^\infty}. 
  \end{align*}
  A similar calculation shows $\Vert \nabla_x f^\bftheta_s - \nabla_x f^\bfphi_s \Vert \leq K_L \Vert \bftheta - \bfphi \Vert_{L^\infty}$. Hence, Gronwall's inequality gives 
  \begin{align*}
      \Vert {[(D\bfx^\bftheta)\bfeta]}_t - {[(Dx^\bfphi)\bfeta]}_t \Vert 
      \leq& K_L \Vert \bfeta \Vert_{L^\infty} \Vert \bftheta - \bfphi \Vert_{L^\infty}. 
  \end{align*}
  Similarly, $(D\bfp^\bftheta)\bfeta$ satisfies the ODE
  \begin{align*}
      &\frac{d}{dt} {[(D\bfp^\bftheta)\bfeta]}_t 
      = - \nabla^2_{xx} H^\bftheta_t {[(D\bfx^\bftheta)\bfeta]}_t 
      - \nabla^2_{xp} H^\bftheta_t {[(Dp^\bftheta)\bfeta]}_t 
      - \nabla^2_{x\theta} H^\bftheta_t \eta_t,
      \\ 
      &{[(D\bfp^\bftheta)\bfeta]}_T = -\nabla^2_{xx} \Phi(x^\bftheta_T, y_0) {[(D\bfx^\bftheta)\bfeta]}_T.
  \end{align*}
  A analogous calculation as above with (A1$''$) shows that 
  \begin{align*}
      \Vert {[(D\bfp^\bftheta)\bfeta]}_t - {[(D\bfp^\bfphi)\bfeta]}_t \Vert 
      \leq& K_L \Vert \bfeta \Vert_{L^\infty} \Vert \bftheta - \bfphi \Vert_{L^\infty}. 
  \end{align*}
\end{proof}

\begin{lemma}
\label{lem:lipschitz_DF}
  Let $h:\R^d\times\R^d \times \Theta \rightarrow \R^m$ have bounded and Lipschitz derivatives in all arguments and define the mapping $\bftheta \mapsto \bfG(\bftheta)$ where ${[\bfG(\bftheta)]}_t = h(x^\bftheta_t, p^\bftheta_t, \bftheta_t)$. Then, $\bfG$ is differentiable and $D\bfG$ is bounded and Lipschitz $\mu_0$-a.s., i.e.
  \begin{align*}
      \Vert D\bfG(\bftheta) \Vert_{L^\infty} 
      &\leq K_B, \\
      \Vert D\bfG(\bftheta) - D\bfG(\bfphi) \Vert_{L^\infty} 
      &\leq K_L \Vert \bftheta - \bfphi \Vert_{L^\infty}.
  \end{align*}
  for some $K_B, K_L>0$ and all $\bftheta,\bfphi\in L^\infty([0,T],\Theta)$. 
\end{lemma}
\begin{proof}
  Let $\bfeta \in L^\infty([0,T],\R^m)$ such that $\Vert \bfeta \Vert_{L^\infty} \leq 1$. 
  By assumptions on $h$ and Lemmas~\ref{lem:lipschitz_xp} and~\ref{lem:lipschitz_Dxp}, $D\bfG$ exists and by the chain rule,
  \begin{align*}
      {[(D\bfG(\bftheta))\bfeta]}_t = \nabla_x h^\bftheta_t {[(D\bfx^\bftheta)\bfeta]}_t
      + \nabla_p h^\bftheta_t {[(D\bfp^\bftheta)\bfeta]}_t
      + \nabla_\theta h^\bftheta_t {\eta}_t,
  \end{align*}
  Thus, $\Vert {[(D\bfG(\bftheta))\bfeta]}_t \Vert \leq K_B \Vert \bfeta \Vert_{L^\infty}$ and
  \begin{align*}
      \Vert {[(D\bfG(\bftheta))\bfeta]}_t - {[(D\bfG(\bfphi))\bfeta]}_t \Vert
      \leq& K_B \Vert \nabla_x h^\bftheta_t - \nabla_x h^\bfphi_t \Vert \\ 
      &+ K_L \Vert {[(Dx^\bftheta)\bfeta]}_t - {[(Dx^\bfphi)\bfeta]}_t \Vert \\ 
      &+ \ldots
  \end{align*}
  The other terms are split similarly and we omit them for simplicity. Using Lipschitz assumption of the derivatives of $h$ and Lemmas~\ref{lem:lipschitz_xp} and~\ref{lem:lipschitz_Dxp}, we obtain the result. 
\end{proof}

Applying Lemma~\ref{lem:lipschitz_DF} with $h=H$ for each sample $i$ and summing, we see that $D\bfF_N$ is bounded and Lipschitz $\mu_0$-a.s.\,and so (B3) is satisfied. It remains to check (B2). Using Lemma~\eqref{lem:lipschitz_DF} and (A1$''$), $\Vert \bfF_N \Vert_{L^\infty}$ and $\Vert D\bfF_N \Vert_{L^\infty}$ are almost surely bounded, hence they satisfy standard concentration estimates. We have:

\begin{lemma}
\label{lem:concentration}
  There exist constants $K_B, K_L>0$ such that for all $\bftheta,\bfphi\in L^\infty([0,T],\Theta)$
  \begin{align*}
      &\P [ \Vert \bfF(\bftheta) - \bfF_N(\bftheta) \Vert_{L^\infty} \geq s] 
      \leq 2 \exp{\left( -\frac{N s^2}{K_1 + K_2 s} \right)}, \\  
      &\P [ \Vert D\bfF(\bftheta) - D\bfF_N(\bftheta) \Vert_{L^\infty} \geq s] 
      \leq 2 \exp{\left( -\frac{N s^2}{K_1 + K_2 s} \right)}.
  \end{align*}
\end{lemma}
\begin{proof}
  Since $\Vert \bfF(\bftheta) \Vert$ is uniformly bounded by $K_B$, we can apply the infinite-dimensional Hoeffding's inequality (\cite{pinelis1986remarks}, Corollary 2) to obtain
  \begin{align*}
      \P [ \Vert \bfF(\bftheta) - \bfF_N(\bftheta) \Vert_{L^\infty} \geq s] 
      \leq 2 \exp{\left( -\frac{N s^2}{2 K_B^2 + (2/3)K_B s} \right)}.
  \end{align*}
  and similarly for $D\bfF_N$.
\end{proof}

Given the above results, we can deduce Thm.~\ref{thm:emp_to_population_F} directly.
\begin{theorem}
\label{thm:emp_to_population_F}
  Let $\bftheta^*$ be a solution $\bfF=0$ (defined in~\eqref{eq:F_defn}), which is stable on $S_\rho(\bftheta^*)$ for some $\rho>0$. Then, there exists positive constants $s_0,C,K_1,K_2$ and $\rho_1<\rho$ and a random variable $\bftheta^N \in S_{\rho_1}(\bftheta^*) \subset L^\infty([0,T],\Theta)$, such that
  \begin{align*}
      \P [ \Vert \bftheta - \bftheta^N \Vert_{L^\infty} \geq C s] 
      &\leq 4 \exp{\left( -\frac{N s^2}{K_1 + K_2 s} \right)}, \qquad s\in(0,s_0], \\ 
      \P [ \bfF_N(\bftheta^N) \neq 0 ] 
      &\leq 4 \exp{\left( -\frac{N s_0^2}{K_1 + K_2 s_0} \right)}.
  \end{align*}
  In particular, $\bftheta^N \rightarrow \bftheta^*$ and $\bfF_N(\bftheta^N)\rightarrow 0$ in probability. 
\end{theorem}

\begin{proof}
    Use Thm.~\ref{thm:banach_finite_sample} with estimates derived in Lemmas~\ref{lem:lipschitz_DF} and~\ref{lem:concentration}. 
\end{proof}

Thm.~\ref{thm:emp_to_population_F} describes the convergence of a solution of the first order condition of the PMP solution in the sampled situation to the population solution of the PMP. Together with a condition of local strong concavity, we show further in Cor.~\ref{cor:emp_to_population_H} that this stationary solution is in fact a local/global maximum of the sampled PMP. The claim regarding the convergence of loss function values is provided in Cor.~\ref{cor:emp_to_population_J}. 
\begin{corollary}
\label{cor:emp_to_population_H}
  Let $\bftheta^*$ be a solution of the mean-filed PMP such that there exists $\lambda_0>0$ satisfying that for a.e. $t\in[0,T]$, $\E \nabla^2_{\theta\theta}H(x_t^{\bftheta^*},p_t^{\bftheta^*},\theta^*_t)+\lambda_0 I  \preceq 0$. Then the random variable $\bftheta^N$ defined in Thm.~\ref{thm:emp_to_population_F} satisfies, with probability at least $1 - 6\exp{[ -{(N \lambda_0^2)}/{(K_1 + K_2 \lambda_0)}]}$, that $\theta^N_t$ is a strict local maximum of sampled Hamiltonian $\frac{1}{N}\sum_{i=1}^N H(x^{\bftheta^N,i}_t, p^{\bftheta^N,i}_t, \theta)$. In particular, if the finite-sampled Hamiltonian has a unique local maximizer, then $\bftheta^N$ is a solution of the sampled PMP with the same high probability.
\end{corollary}

\begin{proof}
  Let
  \begin{align*}
    & [\bfI(\bftheta)]_t \coloneqq \E_{\mu_0} \nabla^2_{\theta\theta}H(x_t^{\bftheta},p_t^{\bftheta},\theta_t),\\
      & [\bfI_N(\bftheta)]_t \coloneqq \frac{1}{N} \sum_{i=1}^N \nabla^2_{\theta\theta}H(x_t^{\bftheta,i},p_t^{\bftheta,i},\theta_t).
  \end{align*}
  Given the assumption of negative definite Hessian matrix at $\theta^*_t$:
  \begin{align*}
    [\bfI(\bftheta^*)]_t + \lambda_0 I \preceq 0,
  \end{align*}
  what we need to prove is
  \begin{align*}
    \PP [\|\bfI_N(\bftheta^N) - \bfI(\bftheta^*) \|_{L^\infty} \geq 2c\lambda_0] \leq o(1), \quad N\rightarrow \infty,
  \end{align*}
  for sufficient small $c>0$. Consider the following estimate
  \begin{align*}
    &\PP [\|\bfI_N(\bftheta^N) - \bfI(\bftheta^*) \|_{L^\infty}  \geq 2c\lambda_0] \\ 
    \leq 
      &\PP [\|\bfI_N(\bftheta^N) - \bfI_N(\bftheta^*) \|_{L^\infty}  \geq c\lambda_0 \text{ and } \|\bfI_N(\bftheta^*) - \bfI(\bftheta^*) \|_{L^\infty}  \geq c\lambda_0]\\
    \leq &\PP [\|\bfI_N(\bftheta^N) - \bfI_N(\bftheta^*) \|_{L^\infty}  \geq c\lambda_0] + \PP [\|\bfI_N(\bftheta^*) - \bfI(\bftheta^*) \|_{L^\infty}  \geq c\lambda_0].
  \end{align*}
  To bound the first term, we can use similar steps as in the proof of Lemma~\ref{lem:lipschitz_DF}, which gives
  \begin{align*}
    \esssup_{t\in[0,T]} 
    \|\nabla^2_{\theta\theta}H(x^{\bftheta}_t,p^{\bftheta}_t,\theta_t) 
    - \nabla^2_{\theta\theta}H(x^{\bfphi}_t,p^{\bfphi}_t,\phi_t)\| \leq
      K_L \|\bftheta-\bfphi\|_{L^\infty}.
  \end{align*}
  Hence we have
  \begin{align*}
    \PP [\|\bfI_N(\bftheta^N) - \bfI_N(\bftheta^*) \|_{L^\infty}  \geq c\lambda_0] 
    \leq & \PP [\|\bftheta^N - \bftheta^*\|_{L^\infty}  \geq c\lambda_0/K_L] \\
    \leq & 4 \exp{\left( -\frac{N \lambda_0^2}{K_1 + K_2 \lambda_0} \right)}.
  \end{align*}
  To bound the second term, note that $\Vert \bfI_N(\bftheta) \Vert$ is uniformly bounded, we can apply the infinite-dimensional Hoeffding's inequality (\cite{pinelis1986remarks}, Corollary 2) to obtain
  \begin{align*}
    \PP [\|\bfI_N(\bftheta^*) - \bfI(\bftheta^*) \|_{L^\infty}  \geq c\lambda_0] \leq 2 \exp{\left( -\frac{N \lambda_0^2}{K'_1 + K'_2 \lambda_0} \right)}.
  \end{align*}
  Combining two estimates together, we complete the proof.
\end{proof}

\begin{corollary}
\label{cor:emp_to_population_J}
  Let $\bftheta^N$ be as defined in Thm.~\ref{thm:emp_to_population_F}. Then there exist constants $K_1, K_2$ such that,
  \begin{align*}
      \PP[|J(\bftheta^N) - J(\bftheta^*)| \geq s] \leq 4 \exp{\left( -\frac{N s^2}{K_1 + K_2s} \right)},
      \qquad s\in(0,s_0].
  \end{align*}
\end{corollary}

\begin{proof}
  Note that $J(\bftheta) = \Phi(x^\bftheta_T,y_0) + \int_{0}^{T} L(x^\bftheta_t, \theta_t) dt$. Using Lemma~\ref{lem:lipschitz_xp}, we have
  \begin{align*}
      |J(\bftheta^N) - J(\bftheta^*)|
      \leq& K_L \Vert x^{\bftheta^*}_T - x^{\bftheta^N}_T \Vert
      + K_L \int_{0}^{T} \Vert x^{\bftheta^*}_t - x^{\bftheta^N}_t \Vert + \Vert \theta^*_t - \theta^N_t \Vert dt \\ 
      \leq& K_L' \Vert \bftheta^N - \bftheta^* \Vert_{L^\infty}
  \end{align*}
  Thus, using Thm.~\ref{thm:emp_to_population_F}, we have 
  \begin{align*}
      \PP [ |J(\bftheta^N) - J(\bftheta^*)| \geq s ] &\leq 
      \PP [ \Vert \bftheta^N - \bftheta^* \Vert_{L^\infty} \geq s/K_L'] \\ 
      &\leq 4 \exp{\left( -\frac{N s^2}{K_1 + K_2s} \right)}.
  \end{align*}
\end{proof}

Thm.~\ref{thm:emp_to_population_F} and Cor.~\ref{cor:emp_to_population_H} establishes a rigorous connection between solutions of the mean-field PMP and its sampled version: when a solution of the mean-field PMP $\bftheta^*$ is stable, then for large $N$, with high probability we can find in its neighborhood a random variable $\bftheta^N$ that is a stationary solution of the sampled PMP~\eqref{eq:compact_PMP_max_sampled}. If further that the maximization is non-degenerate (local concavity assumption in Thm.~\ref{cor:emp_to_population_H}) and unique, then $\theta^N_t$ maximizes the sample Hamiltonian with high probability. Note that this concavity condition is local in the sense that it only has to be satisfied at the paths involving $\bftheta^*$, whereas the strong concavity condition required in Thm.~\ref{thm:small_T_uniqueness} is stronger as it is global. Of course, in the case where the Hamiltonian is quadratic in $\theta$, i.e.\,when $f(x,\theta)$ is linear in $\theta$ and the regularization $L(x,\theta)$ is quadratic in $\theta$ (this is still a nonlinear network, see Example \ref{ex:linear_with_theta}), then all concavity assumptions in the preceding results are satisfied. 

The key assumption for the results in this section is the stability condition (c.f.\,Def.~\ref{def:stability}). In general, this is different from the assumption that $H(x^{\bftheta^*}_t,p^{\bftheta^*}_t,\theta^*_t)$ is strongly concave point-wise in $t$. However, note that one can show using triangle inequality and estimates in Lemma~\ref{lem:lipschitz_Dxp} that if $H$ is strongly concave with sufficiently large concavity parameter ($\lambda_0$), then the solution must stable.  Intuitively, the stability assumption ensures that we can find a small region around $\bftheta^*$ such that it is isolated from other solutions, and this then allows us to find a nearby solution of the sampled problem that is close to this solution. On the other hand, if $DF(\bftheta^*)$ has a non-trivial kernel, then one cannot expect to construct a $\bftheta^N$ that is close to $\bftheta^*$ itself, or any specific point in the kernel. However, one may still find $\bftheta^N$ that is close to the whole kernel. 

Cor.~\ref{cor:emp_to_population_J} is a simple consequence of the previous results, and is effectively a statement about generalization error of the learning model, because it quantifies the difference between loss function values when evaluated on either the population or empirical risk minimization solution. 
We mention an interesting point of the optimal control framework alluded to earlier in the context of generalization. Notice that since we have only assumed that the controls or weights $\bftheta$ are measurable and essentially bounded (and thus can be very discontinuous) in time, we are always dealing with the case where the number of parameters are infinite. Even in this case, we can derive non-trivial generalization estimates. This is to be contrasted with classical generalization bounds based on measures of complexity~\cite{friedman2001elements}, where the number of parameters adversely affect generalization. Note that there are many recent works which take on such issues from varying angles, e.g.,~\cite{neyshabur2017exploring,dziugaite2017computing,arora2018stronger}.

\section{Conclusion}
\label{sec:conclusion}

In this paper, we introduce the mathematical formulation of the population risk minimization problem of continuous-time deep learning in the context of mean-field optimal control. In this framework, the compositional structure of deep neural networks is explicitly taken into account as the evolution of the dynamical system in time. 
To analyze this mean-field optimal control problem, we proceed from two parallel but interrelated perspectives, namely the dynamic programming approach and the maximum principle approach. 
In the former, an infinite-dimensional Hamilton-Jacobi-Bellman (HJB) equation for the optimal loss function values is derived, with state variables being the joint distribution of input-target pairs. The viscosity solution of the derived HJB equation provides us with a complete characterization of the original population risk minimization problem, giving both the optimal loss function value and a optimal feedback control policy. 
In the latter approach, we prove a mean-field Pontryagin's maximum principle that constitutes necessary conditions for optimality. This can be viewed as a local characterization of optimal trajectories, and indeed we formally show that the PMP can be derived from the HJB equation using the method of characteristics. Using the PMP, we study a sufficient condition for which the solution of the PMP is unique. Lastly, we prove an existence result of sampled PMP solutions near the stable solutions of the mean-field PMP. We show how this result connects with generalization errors of deep learning and provide a new direction for obtaining generalization estimates in the case of infinite number of parameters and finite number of sample points. Overall, this work establishes a concrete mathematical framework from which novel ways to attack the pertinent problems in practical and theoretical deep learning may be further developed. 

As a specific motivation for future work, notice that here, we have assumed that the state dynamics $f$ is independent of distribution law of $x_t$  and only depends on $x_t$ itself and control $\theta_t$. There are also more complex network structures used in practice which are beyond this assumption. Let us take batch normalization as an example~\cite{ioffe2015batch}. A batch normalization step involves normalizing inputs using some distribution $\nu$, and then rescale (and re-center) the output using trainable variables so that the matching space is recovered. This has been found empirically to have a good regularization effect for training, but theoretical analysis of such effects are limited. In the present setting, we can write a batch normalization operation as
\begin{equation*}
  BN_{\gamma,\beta}(x,\nu)\coloneqq \gamma \odot \frac{x-\int z\,d\nu(z)}{\sqrt{(z-\int z'\,d\nu(z'))^2d\nu(z) + \epsilon}} + \beta. 
\end{equation*} 
Here $\gamma,\beta \in \R^d$ are trainable parameters, $\odot$ denotes element-wise multiplication, and $\epsilon$ is a small constant avoiding division by zero. Suppose we insert a Batch Normalization operation immediately after the skip connection, the corresponding state dynamics $f$ becomes
\begin{equation*}
  f(x,\theta) \rightarrow f(BN_{\gamma,\beta}(x,\nu),\theta).
\end{equation*}
By incorporating $\gamma,\beta$ into the parameter vector $\theta$ and taking $\nu$ as the population distribution of the state, the equation of state dynamics has the following abstract form
\begin{equation}
  \dot{x}_t=\tilde{f}(x_t,\theta,\P_{x_t}).
  \label{eq:meanfield_forward}
\end{equation}
This is a more general formulation typically considered in the mean-field optimal control literature. The associated objective is very similar to~\eqref{eq:mean_field_ctrl_prob} except the state dynamics:
\begin{align}
    \begin{split}
        \inf_{\bftheta \in L^{\infty}([0,T],\Theta)}              
        J(\bftheta) &\coloneqq \E_{\mu_0} 
        \left[    
        \Phi(x_T, y_0)                        
        + \int_{0}^{T} L(x_t, \theta_t) dt 
        \right],\\ 
        & \text{Subject to~\eqref{eq:meanfield_forward}}. 
    \end{split}
    \label{eq:mean_field_ctrl_prob_general} 
\end{align}
The dynamic programming principle and the maximum principle are still applicable in this setting. For instance, the associated HJB equation can be derived as
\begin{equation*}
\begin{cases}
\displaystyle{\frac{\partial v}{\partial t} + \inf_{\theta\in\Theta}\left\langle \partial_\mu v(t,\mu)(.)\cdot \bar{f}(.,\theta,\mu)+ \bar{L}(.,\theta),\,\mu\right\rangle  = 0,}&\text{on~~} [0,T) \times \cP_2(\R^{d+l}),\\
\displaystyle{v(T, \mu)=\langle\bar{\Phi}(.),\mu\rangle}, &\text{on~~} \cP_2(\R^{d+l}),
\end{cases}
\end{equation*}
where $\bar{f}(w,\theta,\mu)\coloneqq (\tilde{f}(x,\theta,\mu_x),0)$.
Similarly, we expect the following mean field PMP (in the lifted space) to hold under suitable conditions:
\begin{align*}
      & \dot{x}^*_t = \tilde{f}(x^*_t, \theta^*_t, \P_{x^*_t}),
      & & x^*_t = x_0,  \\ 
      & \dot{p}^*_t = - \nabla_x H(x^*_t, p^*_t, \theta^*_t, \P_{x^*_t}),
      & & p^*_T = -\nabla_x \Phi(x^*_T, y_0), \\
  & \E_{\mu_0} H(x^*_t, p^*_t, \theta^*_t, \P_{x^*_t})
      \geq \E_{\mu_0} H(x^*_t, p^*_t, \theta, \P_{x^*_t}),
      & &
      \forall\,\theta \in \Theta, \quad a.e.\,t\in[0,T], 
\end{align*}
where the Hamiltonian function $H: \R^d \times \R^d \times \Theta \times \cP_2(\R^d) \rightarrow \R$ is given by 
\begin{equation*}
  H(x,p,\theta,\mu) = p\cdot f(x, \theta, \mu) - L(x, \theta).
\end{equation*}
Thus, batch normalization can be viewed as a general form of mean-field dynamics, and can be treated in a principled way under the mean-field optimal control framework. We leave the study of further implications of this connection on the theoretical understanding of batch normalization to future work. 

\section*{Acknowledgments}
The work of W. E and J. Han are supported in part by ONR grant N00014-13-1-0338 and Major Program of NNSFC under grant 91130005. Q. Li is supported by the Agency for Science, Technology and Research, Singapore.

\bibliographystyle{unsrt}
\bibliography{ref}

\end{document}